\documentclass[a4paper,11pt]{amsart}
\usepackage{amssymb,amscd}
\usepackage{mathrsfs} 
\usepackage{tikz-cd}
\usepackage{enumitem}
\usepackage[capitalize]{cleveref}
\newcounter{alph}

\newtheorem{theo}[alph]{Theorem}
\newtheorem{coro}[alph]{Corollary}
\numberwithin{equation}{section}
\newtheorem{cor}[equation]{Corollary}
\newtheorem{lem}[equation]{Lemma}
\newtheorem{prop}[equation]{Proposition}
\newtheorem{thm}[equation]{Theorem}

\theoremstyle{definition}

\newtheorem{exa}[equation]{Example}

\def\C{\mathbb C}
\def\N{\mathbb N}
\def\R{\mathbb R}
\def\H{\mathbb H}
\def\O{\mathbb O}
\def\A{\mathcal A}

\def\ve{\varepsilon}
\def\vf{\varphi}
\def\la{\langle}
\def\ra{\rangle}

\newcommand{\ess}{\operatorname{ess}}
\newcommand{\diam}{\operatorname{diam}}
\newcommand{\dive}{\operatorname{div}}
\newcommand{\dv}{\operatorname{dv}}
\newcommand{\dt}{\operatorname{dt}}
\newcommand{\id}{\operatorname{id}}

\newcommand{\inte}{\operatorname{int}}
\newcommand{\orb}{\operatorname{orb}}

\newcommand{\Ric}{\operatorname{Ric}}

\newcommand{\supp}{\operatorname{supp}}

\newcommand{\Lip}{\operatorname{Lip}}
\begin{document}


\title[Bottom of spectra and coverings of orbifolds]{Bottom of spectra and coverings \\ of orbifolds}
\author{Werner Ballmann}
\address
{WB: Max Planck Institute for Mathematics,
Vivatsgasse 7, 53111 Bonn}
\email{hwbllmnn\@@mpim-bonn.mpg.de}
\author{Panagiotis Polymerakis}
\address{PP: Max Planck Institute for Mathematics,
Vivatsgasse 7, 53111 Bonn}
\email{polymerp\@@mpim-bonn.mpg.de}

\dedicatory{Dedicated to Shiing-shen Chern, a great mathematician and a great man}

\thanks{\emph{Acknowledgments.}
We are grateful to the Max Planck Institute for Mathematics and the Hausdorff Center for Mathematics in Bonn for their support and hospitality.}

\date{\today}


\begin{abstract}
We discuss the behaviour of the bottom of the spectrum of scalar Schr\"odinger operators
under Riemannian coverings of orbifolds.
We apply our results to geometrically finite and to conformally compact orbifolds.
\end{abstract}

\maketitle

\tableofcontents

\section{Introduction}
\label{intro}

Spectral invariants of orbifolds are a classical issue in number theory,
but have not yet attracted so much attention in Riemannian geometry.
The prime examples in number theory are closed or finite volume quotients
of Riemannian symmetric spaces of non-compact type
with the modular surface as the most classical one.
Our investigations were motivated by our article \cite{BP2},
in which the bottom of the spectrum of an orbifold quotient occurred in one of the applications
and where the value of that number was in question
(answered by \cref{amen}.\ref{tame} below).
Most of our results are known in the case of manifolds,
the point of the present article is their extension to orbifolds.
These extensions do not come for free,
and even some of the fundamentals have to be prepared appropriately.

We consider a covering $p\colon O_1\to O_0$ of Riemannian orbifolds,
where $O_0$ is connected,
a scalar Schr\"odinger operator $S_0=\Delta+V_0$ on $O_0$ and its lift $S_1$ to $O_1$.
We assume throughout that $S_0$ is bounded from below (on $C^\infty_c(O_0)$).
The orbifold fundamental group $\Gamma_0=\pi_1^{\orb}(O_0)$ of $O_0$ acts on the fibers of $p$,
and we say that the covering is \emph{amenable}
if the action on fibers over regular points of $O_0$ is amenable.
It is noteworthy that a normal covering between connected orbifolds is amenable if and only if its deck transformation group is amenable.
The reader not familiar with these notions is referred to the body of the text below,
where we discuss them in some detail.

 For a Lipschitz function $f\ne0$ with compact support on an orbifold $O$
 and a scalar Schr\"odinger operator $S=\Delta+V$ on $O$,
\begin{align}\label{rayq}
	R_{S}(f) = \frac{\int_{O} (|\nabla f|^2+Vf^2)}{\int_{O}f^2}
\end{align}
is called the \emph{Rayleigh quotient} of $f$ (with respect to $S$).
We then call
\begin{align}\label{bspec}
	\lambda_0(S,O) = \inf R_S(f)
\end{align}
the \emph{bottom of the spectrum of $S$},
where the infimum is taken over all non-zero $f\in C^\infty_c(O)$,
or, equivalently, over all non-zero Lipschitz functions $f$ on $O$ with compact support.
In the case of the Laplacian, we also write $R(f)$ and $\lambda_0(O)$ instead of $R_{\Delta}(f)$ and $\lambda_0(\Delta,O)$, respectively.

If $S$ is bounded from below, that is, $\lambda_0(S,O)>-\infty$,
then the Friedrichs extension $\bar S$ of $S$ in $L^2(O)$
is defined and $\lambda_0(S,O)$ is equal to the bottom of the spectrum $\sigma(S,O)$ of $\bar S$.
To state our main result, we will also need the notion of
the \emph{bottom $\inf\sigma_{\ess}(S,O)$ of the essential spectrum of $S$}.
It is given by
\begin{align}\label{bespec}
	\lambda_{\ess}(S,O) = \sup \lambda_0(S,O\setminus K),
\end{align}
where the supremum is taken over all compact subsets $K$ of $O$.

In the case of coverings as above, we always have that
\begin{align}\label{monot}
	\lambda_0(S_1,O_1) \ge \lambda_0(S_0,O_0),
\end{align}
as we will see in \cref{secele}.
This result is known in many cases; see our survey article \cite{BP1} for references.

By \eqref{monot}, $S_1$ is also bounded from below
so that both, $\lambda_{0}(S_1,O_1)$ and $\lambda_{0}(S_0,O_0)$,
realize the bottom of the spectrum of their respective Friedrichs extension. 

\begin{theo}\label{amen}
Let $p\colon O_1\to O_0$ be a Riemannian covering of orbifolds, where $O_0$ is connected,
and $S_0$ and $S_1$ be compatible scalar Schr\"odinger operators on $O_0$ and $O_1$, respectively,
where $S_0$ is bounded from below.
Then we have:
\begin{enumerate}
\item\label{tame}
if $p$ is amenable, then $\lambda_0(S_1,O_1)=\lambda_0(S_0,O_0)$;
\item\label{name}
if $\lambda_0(S_1,O_1)=\lambda_0(S_0,O_0)<\lambda_{\ess}(S_0,O_0)$,
then $p$ is amenable.
\end{enumerate}
\end{theo}

The study of the behaviour of $\lambda_0$ under Riemannian coverings was initiated by Brooks \cite{Br1,Br2}.
For connected Riemannian manifolds, \cref{amen}.\ref{tame} and \ref{amen}.\ref{name}
are \cite[Theorem 1.2]{BMP1} and \cite[Theorem 4.1]{Po2}.
(The latter is also \cite[Theorem 1.2]{Po3}.)
For further comments and references, we refer to \cite{BP1}.

Note that we do not assume that $O_1$ is connected.
This is important in our proof, where the case that $O_1$ is--possibly--not connected
occurs at an intermediate stage,
but it also seems important in some applications.
We do assume, however, that the orbifolds considered here are second countable
so that they have at most countably many connected components.

\begin{coro}\label{corob}
If $O_0$ contains a compact domain $K$ such that the fundamental groups
of the connected components of the complement $O_0\setminus K$ are amenable,
then there are the following two cases:
\begin{enumerate}
\item\label{ckka}
if $\lambda_0(S_0,O_0)<\lambda_{\ess}(S_0,O_0)$,
then $\lambda_0(S_1,O_1) = \lambda_0(S_0,O_0)$ if and only if $p$ is amenable;
\item\label{ckkb}
if $\lambda_0(S_0,O_0)=\lambda_{\ess}(S_0,O_0)$, then $\lambda_0(S_1,O_1) = \lambda_0(S_0,O_0)$.
\end{enumerate}
\end{coro}

For manifolds, \cref{corob} is \cite[Corollary D]{BP1}.
Using \cref{amen}, the proof there extends to orbifolds.

Let $O$ be a complete and connected Riemannian orbifold
with sectional curvature $-b^2\le K_O\le-a^2$.
Then $O$ is a quotient $\Gamma\backslash X$,
where $X$ is a simply connected Riemannian manifold
and $\Gamma$ a properly discontinuous group of isometries of $X$.
For such an $O$,
let $\Omega$ be the complement of the limit set of $\Gamma$ in the sphere at infinity of $X$.
Following Bowditch \cite{Bo2}, we say that $O$ and $\Gamma$ are \emph{geometrically finite}
if $\Gamma\backslash(X\cup\Omega)$ has finitely many ends and each of them is parabolic
(in the sense of \cite[Section 5.1]{Bo2}).

We say that a Riemannian orbifold $O$ is \emph{hyperbolic}
if it can be written as a quotient $\Gamma\backslash X$,
where $X$ is one of the hyperbolic spaces $H_F^n$ with $F\in\{\R,\C,\H,\O\}$,
endowed with its canonical Riemannian metric, which is unique up to scale.
We normalize it so that $\max K_X=-1$.
Then the number
\begin{align}\label{volent}
	h_X=m+\dim_\R F-2,
\end{align}
where $m=\dim X=n\dim_\R F$,
is equal to the asymptotic volume growth of $X$.
For any hyperbolic space $X$ and non-compact, geometrically finite orbifold $O = \Gamma\backslash X$,
we have
\begin{align}\label{equal}
	\lambda_{\ess}(O) = \lambda_0(X) = h_X^2/4.
\end{align}
For manifolds, the lower estimate $\lambda_{\ess}(O)\ge h_X^2/4$ 
is contained in Hamen\-st\"adt \cite[p.\,282:\,Corollary]{Ham},
the equality is explained in J.\,Li \cite[Remark 1.2]{Li}.
It is likely that their arguments extend to the orbifold case.
However, the orbifold case is also contained in \cite[Theorem B]{BP3},
which contains corresponding estimates for more general kinds of operators over orbifolds,
including the Hodge-Laplacian on differential forms.

\begin{coro}\label{corest}
Let $p\colon O_1\to O_0$ be a Riemannian covering
of complete and connected Riemannian orbifolds.
Assume that there is a non-compact and geometrically finite hyperbolic orbifold $O_0'=\Gamma\backslash X$
such that $O_0\setminus K$ is isometric to $O_0'\setminus K'$
for some compact domains $K\subseteq O_0$ and $K'\subseteq O_0'$.
Then there are the following two cases:
\begin{enumerate}
\item\label{kka}
if $\lambda_0(O_0)<h_X^2/4$,
then $\lambda_0(O_1) = \lambda_0(O_0)$ if and only if $p$ is amenable;
\item\label{kkb}
if $\lambda_0(O_0)=h_X^2/4$, then $\lambda_0(O_1) = \lambda_0(O_0)$.
\end{enumerate}
\end{coro}

For manifolds, \cref{corest} is contained in \cite{BP1}.
Using the characterization of $\lambda_{\ess}(O_0)$ in \eqref{bespec},
\cref{corest} is an immediate consequence of \cref{amen} and \eqref{equal}.

Consider now a compact and connected Riemannian orbifold $P$ with boundary $\partial P\ne\emptyset$
and a smooth non-negative function $\rho$ on $N$ defining $\partial N$, that is,
\begin{align}
  \partial P = \{\rho=0\} \hspace{2mm}\text{and}\hspace{2mm} \partial_\nu\rho > 0
\end{align}
along $\partial P$,
where $\nu$ denotes the inner unit normal of $P$ along $\partial P$ with respect to the given metric $h$ on $P$.
Let $O$ be the interior of $P$, endowed with the conformally equivalent metric 
\begin{align}\label{mazzeo}
   g = \rho^{-2}h
\end{align}
on $O$.
The metric $g$ is complete since any point in $O$ has infinite distance to $\partial P$ with respect to $g$.
Metrics of this kind were introduced by Mazzeo, who called them \emph{conformally compact}.
In \cite[Theorem 1.3]{Ma} he determines the essential spectrum of the Hodge-Laplacian of conformally compact manifolds.
In particular, for functions he obtains that the essential spectrum of $g$ is $[a^2(m-1)^2/4,\infty)$,
where $a=\min\partial_\nu\rho>0$ and $m=\dim O$.
It seems that his arguments also go through for orbifolds.
However, by an easy argument, we will obtain 
\begin{align}\label{coco}
	\lambda_{\ess}(\tilde O) \le \lambda_{\ess}(O) = (m-1)^2a^2/4,
\end{align}
for any connected and conformally compact Riemannian orbifold of dimension $m$,
where $a=\min\partial_\nu\rho$ as above and $\tilde O$ denotes the universal covering space of $O$.
Note the converse monotonicity in \eqref{coco} in comparison to \eqref{monot}.

\begin{coro}\label{cococ}
Let $p\colon O_1\to O_0$ be a Riemannian covering of orbifolds of dimension $m$.
Assume that $O_0$ is conformally compact with $a=\min\partial_\nu\rho$ as above.
Then we have:
\begin{enumerate}
\item\label{cocoa}
if $\lambda_0(O_0)<(m-1)^2a^2/4$, then $\lambda_0(O_1) = \lambda_0(O_0)$
if and only if $p$ is amenable;
\item\label{cocob}
if $\lambda_0(O_0)=(m-1)^2a^2/4$, then $\lambda_0(O_1) = \lambda_0(O_0)$.
\end{enumerate}
\end{coro}

For manifolds, \cref{cococ} is \cite[Theorem 1.10]{BMP2}.
\cref{cococ} is an immediate consequence of Theorems \ref{amen} and \eqref{coco},
using that $\lambda_0(O_1)\le\lambda_0(\tilde O)\le\lambda_{\ess}(\tilde O)$
for the second assertion.

\subsection{Structure of the paper}
\label{sustru}
In Sections \ref{secor}--\ref{secame}, we discuss the structure of orbifolds;
in particular, what we need about the geometry of Riemannian orbifolds,
the analysis of Schr\"odinger operators on orbifolds, and coverings of orbifolds.
In \cref{suscut},
we discuss also the Bishop-Gromov volume and Cheng eigenvalue comparison theorems.
In the short \cref{secele},
we show the monotonicity \eqref{monot} of the bottom of the spectrum under coverings.
In Sections \ref{secsta}--\ref{secgene}, we present the proof of \cref{amen}.
The proof of \eqref{coco} is contained in \cref{seccc}.
 
\section{Preliminaries on orbifolds}
\label{secor}
Our exposition of orbifolds follows mostly \cite[Chapter 13]{Th};
further good references are \cite[Section G]{BH}, \cite{Ha},
\cite[Chapter 6]{Ka}, \cite[Chapter 13]{Ra}, and \cite[\S 2]{Sc}.

An  \emph{($m$-dimensional orbifold) chart} of a Hausdorff space $O$
consists of an open subset $U$ of $O$,
a connected $m$-dimensional manifold $\hat U$ with boundary (possibly empty),
a finite subgroup $G$ of the diffeomorphism group of $\hat U$,
and a continuous map $\pi\colon\hat U\to U$ such that
$\pi$ induces a homeomorphism $G\backslash\hat U\to U$.
The information about such a chart is captured by the diagram
\begin{equation}\label{chartdia}
\begin{tikzcd}
	\hat U \arrow[d,"\pi"'] \arrow[dr] & \\
	U & G\backslash\hat U  \arrow[l,"\cong"]
\end{tikzcd}
\end{equation}
Note that $G\backslash\hat U$ and $U$ are connected since $\hat U$ is connected.

For a chart $a=(U,\hat U,G,\pi)$ of $O$,
we call $U$ the \emph{domain, }$\hat U$ the \emph{codomain}, $G$ the group,
and $\pi$ the \emph{projection of $a$}.

Since $G$ is finite, $\hat U$ carries a $G$-invariant Riemannian metric.
Hence, given $x\in\hat U$, an element $g\in G$ is determined by its value $gx$ and derivative $dg|_x$.

Two $m$-dimensional charts $(U,\hat U,G,\pi)$ and $(U',\hat U',G',\pi')$ of $O$
are said to be \emph{compatible} if, for any $x\in\hat U$ and $x'\in\hat U'$ with $\pi(x)=\pi'(x')$,
there is a local diffeomorphism $f$ from $\hat U$ to $\hat U'$
(a diffeomorphism between open subsets of $\hat U$ and $\hat U'$),
called a \emph{change of charts}, with $f(x)=x'$ and $\pi' f=\pi$.
If the domain of $f$ is (chosen to be) connected, then $f$ is unique up to composition with a $g\in G'$
\cite[p.\,588: Exercise 1.5.1]{BH}.

Assume from now on that $O$ is endowed with an \emph{($m$-dimensional orbifold) atlas of $O$},
that is, a collection $\A$ of $m$-dimensional charts of $O$
such that the domains of the charts from $\A$ cover $O$
and all the charts of $\A$ are compatible with each other.
Since any two charts, which are compatible with each chart from $\A$,
are also compatible with each other, $\A$ determines a unique maximal atlas $\mathcal S$,
a so-called \emph{orbifold structure},
namely the atlas consisting of all charts compatible with all charts from $\A$.
The pair $(O,\mathcal S)$ is called an \emph{($m$-dimensional) orbifold with boundary}
(possibly empty).
Since $\A$ determines $\mathcal S$ uniquely,
we usually view $O$ together with $\mathcal A$ also as an orbifold.

\emph{Convention}:
We assume throughout that orbifolds with boundary are second countable,
although paracompactness would be sufficient in most places.
Moreover, we follow the corresponding tradition for manifolds
and speak of orbifolds when assuming that their boundaries are empty.

Let $O$ be an orbifold with boundary.
Then for any $y\in O$, chart $a$ of $O$ as above with $y\in U$, and point $x\in\hat U$ with $\pi(x)=y$,
the order $|G_x|$ of the stabilizer $G_x$ of $x$ in $G$ does not depend on $x$ and $a$
and is called the \emph{order of $y$}, denoted by $|y|$.
A point $z$ in $O$ is called \emph{regular} if it has order one,
otherwise it is called \emph{singular}.
The \emph{regular set}, that is, the set $\mathcal R=\mathcal R_O$ of regular points,
 is open and dense in $O$.
The complement $\mathcal S=\mathcal S_O$ is the \emph{singular set}.

The orbifold structure of a manifold $M$ with boundary,
given by the connected components of $M$,
where the corresponding groups $G$ are trivial and the maps $\pi$ the identity,
will be called the \emph{trivial (orbifold) structure (of $M$)}.
In the context of orbifolds, manifolds are always endowed with the trivial structure,
unless specified otherwise.

A map $f\colon O'\to O$ between orbifolds with boundary is called \emph{smooth} if,
for all charts $(U',\hat U',G',\pi')$ of $O'$, $(U,\hat U,G,\pi)$ of $O$,
and points $x'\in\hat U'$ and $x\in\hat U$ with $f(\pi'(x'))=\pi(x)$,
there is a smooth map $\vf$ from a neighborhood of $x'$ in $\hat U'$ to $\hat U$
such that $\pi\vf=f\pi'$.
A map $f\colon M\to O$ between a manifold $M$ with boundary and an orbifold $O$ with boundary is smooth
if it is smooth as a  map between orbifolds with boundary, where $M$ is endowed with its trivial orbifold structure.
It is clear by the definition of smoothness that the rank of a smooth map $f$ is well defined.
In particular, we can speak of regular points and regular values of $f$
and can apply the implicit function and Sard theorems.

Using charts in an analogous way, we define tensor fields on orbifolds with boundary.
For example, a \emph{Riemannian metric} on an orbifold $O$ with boundary
consists of a family of Riemannian metrics on the codomains of the charts from an atlas of $O$
such that the actions of the corresponding groups $G$ on the codomains
and the changes of charts are isometric.
An orbifold with boundary together with a Riemannian metric
is called a \emph{Riemannian orbifold with boundary}.

\subsection{Domains in orbifolds}
\label{susdom}
A connected subset $D$ of an orbifold $O$ is called a \emph{domain with smooth boundary}
if the topological boundary $\partial D$ of $D$ admits a covering
by \emph{adapted charts} $a=(U,\hat U,G,\pi)$ of $O$, that is,
$\hat U=(-\ve,\ve)\times\hat V$ with $G$ acting trivially on $(-\ve,\ve)$,
\begin{align}\label{adacha}
	\pi^{-1}(D\cap U) = (-\ve,0]\times\hat V,
	\hspace{3mm}\text{and}\hspace{3mm}
	\pi^{-1}(\partial D\cap U) = \{0\}\times\hat V.
\end{align}
The vector field $\partial/\partial r$, where $r$ denotes the variable in $(-\ve,\ve)$,
is invariant under $G$ and vanishes nowhere in $\hat U$.
This restricts the nature of points which can be boundary points of domains with smooth boundary.
Clearly, any domain in $O$ with smooth boundary is the sublevel set of a regular value
of a smooth function on $O$.
Note also that $\pi_{W}^{-1}(\partial D\cap W)$ is a submanifold of $\hat W$,
for any chart $b=(W,\hat W,H,\pi_{W})$ of $O$.

The adapted charts turn $\partial D$ into an orbifold,
and regular and singular set of $\partial D$ are equal to
$\mathcal R_{\partial D}=\partial D\cap\mathcal R_O$
and $\mathcal S_{\partial D}=\partial D\cap\mathcal S_O$, respectively.

\section{Riemannian orbifolds}
\label{secror}
As defined above, an orbifold $O$ with boundary together with a Riemannian metric on $O$
is called a Riemannian orbifold with boundary.

\subsection{Distance and completeness}
\label{sudico}
Given a Riemannian orbifold $O$ with boundary,
the length of a piecewise smooth curve is defined to be the sum of the lengths
of the local lifts of the curve to codomains of charts.
The distance $d(x,y)$ of two points $x,y\in O$ is defined
to be the infimum of the lengths of piecewise smooth curves in $O$ joining $x$ to $y$.
Since $O$ is a Hausdorff space,
the distance function $d$ is a metric on $O$ in the standard sense if $O$ is connected.
It is then easy to see that $O$ is an \emph{interior metric space},
that is, $d(x,y)$ is the infimum of the lengths of rectifiable curves in $O$ joining $x$ to $y$.

In what follows, let $O$ be a Riemannian orbifold with boundary.

\begin{lem}\label{hatlift}
Let $(U,\hat U,G,\pi)$ be a chart for $O$ and $c\colon[a,b]\to U$ be a rectifiable curve.
Then there is a lift $\sigma\colon[a,b]\to\hat U$ of $c$ such that $L(\sigma|_{[s,t]})=L(c|_{[s,t]})$,
for all $a\le s,t\le b$.
\end{lem}

\begin{proof}
We can assume without loss of generality that $c$ has unit speed.

For any $\ve>0$, there are a partition $a=t_0<\dots<t_k=b$ of $[a,b]$,
points $p_i\in\hat U$ above $c(t_i)$,
and minimizing geodesics $\sigma_i\colon[t_{i-1},t_i]\to\hat U$ from $p_{i-1}$ to $q_i$
such that $\pi(q_i)=c(t_{i})$ and
\begin{align*}
	\sum L(\sigma_i) = L(c)\pm\ve.
\end{align*}
By rearranging the points $p_i$, we can assume that $q_i=p_{i}$.
Then
\begin{align*}
	\sigma = \sigma_1*\dots*\sigma_k
\end{align*}
is a continuous curve in $\hat U$.
Since the $\sigma_i$ are minimizing, we have
\begin{align*}
	L(\sigma_i)\le L(c|_{[t_{i-1},t_i]})\le t_i-t_{i-1}.
\end{align*}
It is not hard to see that we may apply the Arzela-Ascoli theorem to a sequence of such curves,
where $\ve=\ve_n\to0$, to get a lift $\sigma$ of $c$ as asserted.
\end{proof}

The following result is an immediate consequence of Cohn-Vossen's generalization
of the Hopf-Rinow Theorem \cite{CV} (\cite[Section I.2]{Bal}).

\begin{thm}\label{hori}
If $O$ is connected, then the following are equivalent:
\begin{enumerate}
\item\label{hr1}
$O$ is complete as a metric space;
\item\label{hr2}
any minimizing geodesic $c\colon[0,1)\to O$ can be extended to $[0,1]$;
\item\label{hr3}
for some $x\in O$,
any minimizing geodesic $c\colon[0,1)\to O$ with $c(0)=x$ can be extended to $[0,1]$; 
\item\label{hr4}
bounded subsets of $O$ are relatively compact.
\end{enumerate}
Moreover, each of these properties implies that, for any pair $x,y\in O$,
there is a minimizing geodesic from $x$ to $y$.
\end{thm}

The proof of the following result is close to the proof
of the corresponding result for manifolds. 

\begin{prop}\label{comex}
Any orbifold with boundary admits a complete Riemannian metric.
\end{prop}

\subsection{Riemannian measure}
\label{susior}
Let $O$ be a Riemannian orbifold of dimension $m$.
Then the volume element $\dv$ of  the Riemannian metric of $O$
is well defined on the manifold $\mathcal R$ of regular points of $O$,
and we define the \emph{($m$-dimensional) measure} of a Borel set $B$ of $O$ by
\begin{align}\label{mb}
	|B| =|B|_m = \int_{\mathcal R} \dv(z) \chi_B(z),
\end{align}
where $\chi_B$ denotes the characteristic function of $B$.
To justify this definition,
we observe that $O$ can be covered by the domains of countably many charts of $O$
and that, in the codomain of each chart of $O$, the set of singular points has measure zero.
With \eqref{mb},
we obtain a positive measure $\dv$ on the $\sigma$-algebra of Borel sets of $O$.
Furthermore, if $B\subseteq O$ is a Borel set which is contained in the domain $U$
of a chart $(U,\hat U,G,\pi)$, then
\begin{align}\label{mba}
	|G| |B| = \int_{\hat U} \dv(z) \chi_B(\pi(z)) = |\pi^{-1}(B)|,
\end{align}
where the right hand side denotes the Riemannian volume of $\pi^{-1}(B)$.

Suppose that $D$ is a compact domain with smooth boundary $\partial D$ in $O$
as in \cref{susdom}.
Then adapted charts turn $\partial D$ into an orbifold,
endowed with the induced Riemannian metric.
Hence by \eqref{mb}, applied to $\partial D$,
we obtain an ($(m-1)$-dimensional) measure on the $\sigma$-algebra of Borel sets in $\partial D$,
\begin{align}\label{mdb}
	|B| =|B|_{m-1} = \int_{\partial D\cap\mathcal R} \dv_{m-1}(z) \chi_B(z),
\end{align}
where $\dv_{m-1}$ denotes the induced volume element of the submanifold $\partial D\cap\mathcal R$
in the Riemannian manifold $\mathcal R$.
Here we recall that $\partial D\cap\mathcal R=\mathcal R_{\partial D}$.

Let now $b=(V,\hat V,H,\pi_V)$ be an adapted chart for $\partial D$,
$a=(U,\hat U,G,\pi_U)$ an arbitrary chart of $O$,
$z\in U\cap V$ a singular point, $x\in\pi_U^{-1}(z)$, and $y\in\pi_V^{-1}(z)$.
Then there is a local diffeomorphism $\vf$ respecting the projections
from a neighborhood of $x$ in $\hat U$ to a neighborhood of $y$ in $\hat V$.
In particular, $\vf(x)=y$.
Moreover, $\vf$ respects the order of points
so that $\vf$ maps regular and singular points in $\hat U$ to respective points in $\hat V$.
In particular, since the set of singular points in $\pi_V^{-1}(\partial D)$
has $(m-1)$-dimensional measure zero in $\hat V$,
the same is true for the set of singular points in the submanifold $\pi_U^{-1}(\partial D)$ in $\hat U$.
Therefore, if $B\subseteq\partial D$ is a Borel set which is contained in $U$, then
\begin{align}\label{mdba}
	|G| |B|_{m-1}
	= \int_{\pi_U^{-1}(\partial D)} \dv_{m-1}(z) \chi_B(\pi(z))
	= |\pi_U^{-1}(B)|_{m-1},
\end{align}
where the right hand side denotes the Riemannian volume of $\pi_U^{-1}(B)$
in the manifold $\pi_U^{-1}(\partial D)$ with the induced Riemannian metric.

\begin{lem}[Coarea formula]\label{caf}
For $f\in C^\infty(O)$ and $\vf\in C^0_c(O)$
\begin{align*}
	\int_O |\nabla f|\dv_m \vf = \int_\R \int_{\{f=t\}} \dt \dv_{m-1} \vf.
\end{align*}
\end{lem}

\begin{proof}
Only the regular points of $f$ matter.
In the neighborhood of such a point $x$,
we may introduce coordinates by using the local flow $(F_s)$ of $X=\nabla f/|\nabla f|^2$,
that is, $(s,y)\sim F_s(y)$ with $y\in\{f(y)=t\}$, where $f(x)=t$.
In such coordinates, we have $|\nabla f|\dv_m=\dt \dv_{m-1}$.
\end{proof}

\subsection{Cut locus and two comparison results}
\label{suscut}

Let $O$ be a connected Riemannian orbifold (that is, orbifold with empty boundary).
It is then easy to see that a curve $c\colon I\to O$, that is minimizing locally,
has local lifts to codomains of charts over it and that, up to parametrization,
these local lifts are geodesics in the usual sense.

\begin{prop}[Proposition 15 in \cite{Bor}]\label{borz}
If $c\colon[a,b]\to O$ is a minimal geodesic and $c(t)\in\mathcal S$ for some $t\in(a,b)$,
then $c$ is contained in $\mathcal S$.
\end{prop}

Assume for the rest of this subsection that $O$ is complete.
Let $x\in O$ and $c$ be a (non-constant) geodesic starting at $x$.
Then the first point on $c$ behind which $c$ is not a minimal connection to $x$ anymore
is called the \emph{cut point} of $x$ along $c$.
The set $C(x)$ of all cut points of $x$ along geodesics from $x$
is called the \emph{cut locus of $x$}.

\begin{prop}\label{cutloc}
We have:
\begin{enumerate}
\item\label{cl0}
for any direction $v$ at $x$, the time $t_0(v)>0$,
at which the unit speed geodesic from $x$ in the direction of $v$ stops being minimizing,
depends continuously on $v$;
\item \label{cl1}
for any $y\in O\setminus C(x)$,
there is a unique minimal geodesic from $x$ to $y$;
\item\label{cl2}
$C(x)$ is closed and $|C(x)|_m=0$, where $m=\dim O$;
\item\label{cl3}
if $x\in\mathcal R$, then $\mathcal S\subseteq C(x)$.
\end{enumerate}
\end{prop}

\begin{proof}
Mutatis mutandis,
\eqref{cl0}--\eqref{cl2} follow easily from corresponding arguments in the case of Riemannian manifolds.
\cref{borz} implies \eqref{cl3}.
\end{proof}

The Bishop-Gromov volume and Cheng eigenvalue comparison theorems extend to orbifolds.
Denote by $B_r(k)$ the ball of radius $r$ in $M_k^m$,
the model space of dimension $m=\dim O$ and constant sectional curvature $k$.

\begin{thm}[Bishop-Gromov volume comparison]\label{bgv}
Assume that $\Ric\ge(m-1)k$ on $B_s(x)$.
Then
\begin{align*}
	\frac{|B_s(x)|}{\beta(s)}
	\le \frac{|B_r(x)|}{\beta(r)}
	\xrightarrow[r\to 0]{} \frac1{|x|}
\end{align*}
for all $0<r<s$,
where $\beta(r)=|B_r(k)|$ .
Moreover, equality on the left holds for some $0<r<s$
if and only if $B_s(x)$ is isometric to $G\backslash B_s(k)$,
where $G$ is a finite group of isometries of $M_k^m$ fixing the center of $B_s(k)$.
\end{thm}

Except for the equality discussion, \cref{bgv} is \cite[Proposition 20]{Bor}.

\begin{thm}[Cheng eigenvalue comparison]\label{cheng}
Assume that $\Ric\ge(m-1)k$ on $B_r(x)$.
Then
\begin{align*}
	\lambda_0(B_r(x))\le\lambda_0(B_r(k)).
\end{align*}
Moreover, equality holds if and only if $B_r(x)$ is isometric to $G\backslash B_r(k)$,
where $G$ is a finite group of isometries of $M_k^m$ fixing the center of $B_r(k)$.
\end{thm}

\begin{proof}[Sketch of proofs of Theorems \ref{bgv} and \ref{cheng}]
Recall that the standard proofs of Theorems \ref{bgv} and \ref{cheng}
in the manifold case are obtained via integrating associated inequalities along radial geodesics;
see \cite[Theorem 1.1]{Ch} for \cref{cheng}.
Using \cref{cutloc},
the same procedure leads to the above assertions for orbifolds.
\end{proof}

\subsection{Cheeger constants}
\label{suscc}
The \emph{Cheeger constant} of a Riemannian orbifold $O$ of dimension $m$
is defined to be
\begin{align}\label{chco}
	h(O) = \inf_A \frac{|\partial A|_{m-1}}{|A|_m},
\end{align}
where the infimum is taken over all compact domains $A$ of $O$ with smooth boundary,
and where $|.|_k$ indicates $k$-dimensional Riemannian volume.
The Cheeger constant is related with the bottom of the spectrum of the Laplacian on $O$
via the \emph{Cheeger inequality},
\begin{align}\label{chine}
	\lambda_0(O) \ge \frac{1}{4} h(O)^2.
\end{align}
For the convenience of the reader,
we present a short proof of \eqref{chine}.
Namely, for the Rayleigh quotient $R(f)$ of a non-vanishing function $f\in C^\infty_c(O)$,
\begin{align*}
	\int_O |\nabla f^2|
	&= \int_O 2|f||\nabla f| \\
	&\le 2\left(\int_O f^2\right)^{1/2}\left(\int_O|\nabla f|^2\right)^{1/2}
	\le 2R(f)^{1/2} \int_O f^2,
\end{align*}
by the Schwarz inequality.
On the other hand, by \cref{caf}, the definition of $h(O)$, and Cavalieri's principle,
respectively,
\begin{align*}
	\int_O|\nabla f^2|
	&= \int_0^\infty |\{f^2=t\}|_{m-1} \dt \\
	&\ge h(O) \int_0^\infty |\{f^2\ge t\}|_m \dt
	= h(O) \int_O f^2.
\end{align*}
We conclude that $h(O)^2\le4R(f)^2$.
Now $\lambda_0(O)$ is the infimum of Rayleigh quotients $R(f)$ over non-vanishing $f\in C^\infty_c(O)$,
hence \eqref{chine} follows.

In analogy with \eqref{bespec},
we define the \emph{asymptotic Cheeger constant of $O$} by
\begin{align}\label{bespeca}
	h_{\ess}(O) = \sup h(O\setminus K),
\end{align}
where the supremum is taken over all compact subsets $K$ of $O$.
In \eqref{china}, we obtain an analogue of \eqref{chine},
relating the asymptotic Cheeger constant to the bottom of the essential spectrum of $O$.

Recall that, for a relatively compact open domain $D$ in a Riemannian manifold $M$,
the Cheeger constant of $D$ with respect to Neumann boundary condition is defined to be
\begin{align}\label{chcon}
	h^{N}(D) = \inf \frac{|\partial A \cap D|}{|A|},
\end{align}
where the infimum is taken over all domains $A\subseteq D$ with $|A|\le|D|/2$
and smooth intersection $\partial A \cap D$.

\begin{lem}[Lemma 5.1 of Buser \cite{Bu}]\label{bulem}
If $M$ is of dimension $m$ and complete with Ricci curvature bounded from below by $1-m$,
$D\subseteq M$ is starlike with respect to a point $x\in M$,
and $B_r(x)\subseteq D\subseteq B_R(x)$, then
\begin{align*}
	h^N(D) \ge C_{m}^{1+R} \frac{r^{m-1}}{R^m}.
\end{align*}
\end{lem}

\section{Analysis on orbifolds}
\label{secano}

Throughout this section, let $O$ be a Riemannian orbifold.
Then the Laplace operator $\Delta$ is well defined on $C^\infty(O)$.

\begin{lem}\label{green}
	For functions $f,g\in C^\infty(O)$ and a compact domain $D\subseteq O$ with smooth boundary $\partial D$,
	we have
	\begin{align*}
		\int_D f\Delta g = \int_D \la\nabla f,\nabla g\ra  - \int_{\partial D} f\partial_\nu g.
	\end{align*}
\end{lem}

\begin{proof}
	Using a smooth partition of unity on $O$, the proof reduces to the two cases
	where the support of $f$ is contained in a coordinate domain in the interior of $D$
	or in one for the boundary of $D$.
	We only discuss the less trivial second case,
	where the codomain $\hat U=\hat V\times[0,\ve)$
	with the corresponding finite group $G$ acting trivially on the factor $[0,\ve)$.
	We obtain
	\begin{align*}
		|G|\int_D f\Delta g
		&= \int_{\hat U} \hat f\Delta\hat g \\
		&= \int_{\hat U} \la\nabla\hat f,\nabla\hat g\ra - \int_{\hat V} \hat f\partial_\nu\hat g \\
		&= |G|\int_D \la\nabla f,\nabla g\ra - |G|\int_{\partial D} f\partial_\nu g,
	\end{align*}
	where $\hat f=f\circ \pi$, $\hat g=g\circ \pi$
	and $\nu$ denotes the exterior normal vector field of $\hat U$ and $D$, respectively.
\end{proof}

We consider a Schr\"odinger operator $S=\Delta+V$ on $O$ with smooth potential $V$. In view of \cref{green}, the operator
\[
	S \colon D(S) \subseteq L^{2}(O) \to L^{2}(O)
\]
with domain $D(S)=C^{\infty}_{c}(O)$ is symmetric.
It is evident that $S$ is bounded from below if $V$ is bounded from below.

Assume from now on that $S$ is a Schr\"odinger operator on $O$ that is bounded from below,
and choose $\beta \in \mathbb{R}$ such that
\[
	\langle Sf , f \rangle_{L^{2}} \geq \beta \| f \|_{L^{2}}^{2}
\]
for any $f \in C^{\infty}_{c}(O)$.
Denote by $H_{S} \subseteq L^{2}(O)$ the completion of $C^{\infty}_{c}(O)$ with respect to the inner product
\[
	\langle f , h \rangle_{H_S} = \langle f ,h \rangle_{L^{2}} + \langle (S - \beta) f , h \rangle_{L^{2}}.
\]
Since $\|f\|_{H_S}\ge\|f\|_{L^2}$, we view $H_S\subseteq L^2(O)$.
The Friedrichs extension of $S$ is the self-adjoint operator
\begin{align*}
	\bar S \colon D(\bar S) \subseteq L^{2}(O) \to L^{2}(O)
\end{align*}
with $\bar{S} = S^*$ on its domain $D(\bar{S}) = H_{S} \cap D(S^*)$,
where $S^*$ denotes the adjoint of $S$. 

We denote by $\lambda_0(S)$ the bottom of the spectrum of $\bar S$.
The Rayleigh quotient of a non-zero $f \in H_{S}$ with respect to $S$
is defined by
\begin{align}\label{rayq2}
	R_{S}(f)
	= \frac{\| f \|_{H_{S}}^{2}}{\| f \|_{L^{2}}^{2}} + \beta -1.
\end{align}
It is well known from functional analysis that
\begin{align}\label{bspec2}
	\lambda_{0}(S) = \inf R_{S}(f),
\end{align}
where the infimum is taken over all non-zero $f \in C^{\infty}_{c}(O)$
or over all non-zero $f \in H_{S}$.

\begin{lem}\label{Lipschitz}
	Any compactly supported Lipschitz function on $O$ belongs to $H_{S}$.
\end{lem}

\begin{proof}
	It suffices to prove the assertion for any compactly supported Lipschitz function  $f$
such that there exists a coordinate chart $(U, \hat{U} ,G, \pi)$ with $U$ precompact 
and $\supp f \subseteq U$.
Setting $\hat{f} := f \circ \pi$, we compute
	\[
	| \hat{f}(x) - \hat{f}(y) | = |f(\pi(x)) - f(\pi(y))| \leq C d(\pi(x) , \pi(y)) \leq C d(x,y),
	\]
	where $x,y \in \hat{U}$ and $C$ is the Lipschitz constant of $f$.
	This yields that $\hat{f}$ is a Lipschitz function, and, therefore,
	there exists a sequence $(\hat{f}_{n})_{n \in \mathbb{N}}$ in $C^{\infty}_{c}(\hat{U})$ 	
	with $\hat{f}_{n} \rightarrow \hat{f}$ in $H_{0}^{1}(\hat{U})$.
	It is evident that the functions
	\[
	\hat{h}_{n} := \frac{1}{|G|} \sum_{g \in G} \hat{f}_{n} \circ g
	\]
	descend to functions $h_{n} \in C^{\infty}_{c}(O)$.
	Using that $\hat{f}$ is $G$-invariant,
	we have that $\hat{h}_{n} \rightarrow \hat{f}$ in $H_{0}^{1}(\hat{U})$
	and, hence, that $h_{n} \rightarrow f$ in $H_{0}^{1}(U)$.
	Since the $h_n$ are smooth,
\begin{align*}
	\|h_n-h_m\|_{H_S}^2
	\leq (1-\beta + C_{V})\|h_n-h_m\|_{L^2}^2 + \|\nabla h_n-\nabla h_m\|_{L^2}^2,
\end{align*}
where $C_{V}$ is the supremum of $|V|$ on $U$.
Therefore $h_{n} \rightarrow f$ in $H_{S}$, as we wished.
\end{proof}

As an immediate consequence of \cref{Lipschitz},
we obtain the following assertion from the introduction (see \eqref{bspec}).

\begin{cor}\label{bspec3}
The bottom $\lambda_0(S)$ of the spectrum of $\bar S$ is given by
\begin{align*}
		\lambda_0(S) = \inf R_S(f),
\end{align*}
where the infimum is taken over all non-zero $f \in \Lip_{c}(O)$.
\end{cor}

\begin{lem}\label{eigen}
	Suppose that $O$ is connected and that $\lambda_{0}(S)$ is an eigenvalue of $\bar{S}$.
	Then any eigenfunction $\vf$ of $\bar{S}$ corresponding to $\lambda_{0}(S)$ is smooth and nowhere vanishing.
	In particular, the eigenspace corresponding to $\lambda_{0}(S)$ is one-dimensional.
\end{lem}

\begin{proof}
	We know from elliptic regularity that any eigenfunction $\vf$ is smooth.
	Since $\vf \in H_{S}$, there exists a sequence $(f_{n})_{n \in \mathbb{N}}$ in $C^{\infty}_{c}(O)$
	with $f_{n} \rightarrow \vf$ in $H_{S}$.
	Then the sequence $(|f_{n}|)_{n \in \mathbb{N}}$ belongs to $H_S$ and is bounded in $H_{S}$.
	Hence it has a weakly convergent subsequence in $H_{S}$.
	On the other hand, $|f_{n}| \rightarrow |\vf|$ in $L^{2}(O)$.
	Therefore, $|\vf| \in H_{S}$ and $R_{S}(|\vf|) = \lambda_{0}(S)$,
	which yields that $|\vf|$ is an eigenfunction of $\bar{S}$ corresponding to $\lambda_{0}(S)$,
	and, in particular, that $|\vf|$ is smooth.
	From the maximum principle it follows that $|\vf|$ is positive, that is, $\vf$ is nowhere vanishing.
	The second assertion is a consequence of the first
	since the first implies that there are no $L^{2}$-perpendicular eigenfunctions corresponding to $\lambda_{0}(S)$.
\end{proof}

\begin{lem}\label{min seq}
	Suppose that $O$ is connected and that $\lambda_{0}(S) < \lambda_{\ess}(S)$.
	Let $(f_{n})_{n \in \mathbb{N}}$ be a sequence in $\Lip_{c}(O)$
	with $\| f_{n} \|_{L^{2}} \rightarrow 1$ and $R_{S}(f_{n}) \rightarrow \lambda_{0}(S)$.
	Then there exists a subsequence $(f_{n_{k}})_{k \in \mathbb{N}}$ such that $f_{n_{k}} \rightarrow \vf$
	in $L^{2}(O)$ for some eigenfunction $\vf$ of $\bar{S}$ corresponding to $\lambda_{0}(S)$.
\end{lem}

\begin{proof}
	We know from Lemma \ref{Lipschitz} that any compactly supported Lipschitz function
	can be approximated by compactly supported smooth functions in $H_{S}$.
	Hence, it suffices to prove the assertion for the case where the $f_{n} \in C^{\infty}_{c}(O)$.
	Since $\lambda_{0}(S) < \lambda_{\ess}(S)$,
	we have that $\lambda_{0}(S)$ is an isolated eigenvalue of $\bar{S}$ of finite multiplicity.
	Let $E$ be the corresponding eigenspace, and write $f_{n} = g_{n} + h_{n}$,
	where $g_{n} \in E$ and $h_{n}$ is $L^{2}$-perpendicular to $E$.
	From our assumption, it follows that $R_{S}(h_{n}) \geq \lambda_{0}(S) + c$ for some $c>0$.
	Since $E$ is one-dimensional, after passing to a subsequence if necessary,
	we may assume that $g_{n} \rightarrow \vf$ in $L^{2}(O)$ for some $\vf \in E$. 
	
	Given $\varepsilon > 0$, we have that $R_{S}(f_{n}) < \lambda_{0}(S) + \varepsilon$ for sufficiently large $n \in \mathbb{N}$. Using that $\langle \bar{S}h_{n} , g_{n} \rangle = 0$, we compute
	\begin{align*}
		(\lambda_{0}(S) + c)\|h_{n}\|_{L^{2}}^{2} &\le \langle \bar{S}h_{n} , h_{n} \rangle_{L^{2}} = \langle \bar{S}f_{n} , f_{n} \rangle_{L^{2}} - \langle \bar{S}g_{n} , g_{n} \rangle_{L^{2}} \\
		&\le (\lambda_{0}(S) + \varepsilon) \| f_{n} \|_{L^{2}}^{2} - \lambda_{0}(S) \| g_{n} \|_{L^{2}}^{2} \\
		&\le \lambda_{0}(S) \| h_{n} \|_{L^{2}}^{2} + \ve \| f_{n} \|_{L^{2}}^{2}
	\end{align*}
	for sufficiently large $n \in \mathbb{N}$.
	This shows that $h_{n} \rightarrow 0$ and hence that $f_{n} \rightarrow \vf$ in $L^{2}(O)$.
\end{proof}

\begin{prop}\label{comess}
The bottom of the essential spectrum of $S$ is given by
\[
	\lambda_{\ess}(S,O) = \sup\lambda_{0}(S,O\setminus K),
\]
where the supremum is taken over all compact subsets $K\subseteq O$.
\end{prop}

\begin{proof}
By Weyl's criterion, $\lambda\in\R$ belongs to the essential spectrum of $\bar S$
if and only if there exists a \emph{Weyl sequence} $(f_n)_{n\in\N}$ for $\lambda$
in the domain $D(\bar{S})$,
which means that $\|f_n\|_{L^2}\to1$, $f_n\rightharpoonup0$, and $(\bar S-\lambda)f_n\to0$ in $L^2(O)$.

Given such a sequence and a compact subset $K$ of $O$,
we want to show that there exists a cut off function $\chi\in C^\infty_c(O)$ such that,
after passing to a subsequence if necessary,
$((1-\chi)f_n)_{n\in\N}$ is a Weyl sequence for $\lambda$ with supports disjoint from $K$. 
	
It is evident that it suffices to prove this for any compact domain $K$
contained in a coordinate region $U$ of a chart $(U,\hat{U},G,\pi)$ of $O$.
Fix $\chi \in C^{\infty}_{c}(O)$ with $\chi = 1$ in a neighborhood of $K$ and $\supp\chi\subseteq U$.
From elliptic estimates (on $\hat{U}$),
it is not hard to see that $(\hat{\chi}\hat{f}_{n})_{n \in \mathbb{N}}$ is bounded in $H^{2}(\hat{U})$
and that $\hat{\chi}\hat{f}_{n} \rightharpoonup 0$ in $L^{2}(\hat{U})$,
where $\hat\chi=\chi\circ\pi$ and $\hat f_n=f_n\circ\pi$.
After passing to a subsequence if necessary,
this yields that $\hat\chi\hat f_n \rightharpoonup 0$ in $H^2(\hat{U})$.
Since the supports of these functions are contained in a compact subset of $\hat{U}$,
it follows from Rellich's lemma that $\hat\chi\hat f_n\to0$ in $H^1(\hat{U})$.
Therefore $\chi f_n\to0$ and $(\bar{S} - \lambda) (\chi f_n)\to0$ in $L^2(O)$.
We conclude that $((1-\chi)f_{n})_{n \in \mathbb{N}}$ is a Weyl sequence for $\lambda$
with supports disjoint from $K$.
In particular,
\begin{align*}
	\lambda_{\ess}(S,O) \ge \lambda_{0}(S,O\setminus K).
\end{align*}
To finish the proof, we may assume that $\lambda_\infty=\sup\lambda_{0}(S,O\setminus K)$ is finite.
Now for any $\ve>0$, we may choose an exhaustion of $O$ by compact subsets $K_n$
and a sequence of functions $f_n\in C^\infty_c(O\setminus K_n)$ 
with pairwise disjoint supports such that $\|f_n\|_{L^2}=1$
and $R_S(f_n)\le\lambda_n+\ve$, where $\lambda_n=\lambda_0(S,O\setminus K_n)$.
We see that the space of functions $f\in C^\infty_c(O)$ with Rayleigh quotient at most $\lambda_\infty+2\ve$
is infinite dimensional.
Hence $\lambda_{\ess}(S,O)\le\lambda_\infty$.
\end{proof}

The asymptotic Cheeger inequality 
\begin{align}\label{china}
	\lambda_{\ess}(O) \ge \frac{1}{4} h_{\ess}(O)^2
\end{align}
is an immediate consequence of \eqref{chine} and \cref{comess}.

Consider a positive $\vf \in C^{\infty}(O)$ satisfying $S \vf = \lambda \vf$ for some $\lambda \in \mathbb{R}$. Denote by $L^{2}_{\vf}(O)$ the $L^{2}$-space of $O$ with respect to the measure $\vf^{2} \dv$,
where $\dv$ is the measure induced from the Riemannian metric of $O$.
It is easy to see that the isometric isomorphism $m_{\vf} \colon L^{2}_{\vf}(O) \to L^{2}(O)$,
defined by $m_{\vf} f = f \vf$, intertwines $S - \lambda$ with the diffusion operator
\[
L = m_{\vf}^{-1}  (S - \lambda)  m_{\vf} = \Delta - 2 \nabla \ln \vf.
\]
The operator $L$ is called \emph{renormalization of $S$ with respect to $\vf$}.
The Rayleigh quotient of a non-zero $f \in C^{\infty}_{c}(O)$ is defined by
\[
	R_{L}(f) = \frac{ \langle L f , f \rangle_{L^{2}_{\vf}} }{\| f \|_{L^{2}_{\vf}}^{2}}
	= \frac{ \int_{O} | \nabla f |^{2} \vf^{2}}{\int_{O} f^{2} \vf^{2}}.
\]

\begin{lem}\label{renorm}
For any non-zero $f \in C^{\infty}_{c}(O)$, we have $R_{L}(f) = R_{S}(f \vf) - \lambda$.
\end{lem}

\section{Coverings of orbifolds}
\label{seccov}
A map $p\colon O'\to O$ between orbifolds with boundary is called a \emph{covering (of orbifolds)}
if, for each chart $a=(U,\hat U,G,\pi)$ of $O$ with simply connected domain $U$,
the preimage $p^{-1}(U)$ is the disjoint union of connected open subsets $U'$ of $O'$,
which belong to charts of $O'$ of the form $a'=(U',\hat U,G',\pi')$,
where $G'\subseteq G$ and $G'$ may depend on $a'$,
such that the diagram
\begin{equation}\label{comdia}
\begin{tikzcd}
	\hat U \arrow[r,"\pi'"] \arrow[d,"q"]
	&
	U' \arrow[r,"\cong"] \arrow[d,"p"]
	&
	G'\backslash\hat U \arrow[d]
	\\
	\hat U  \arrow[r,"\pi"]
	&
   	U \arrow[r,"\cong"]
	&
	G\backslash\hat U
\end{tikzcd}
\end{equation}
commutes,
where $q$ denotes the identification of the codomain $\hat U$ of $a'$ with the codomain $\hat U$ of $a$
and the right vertical arrow and each of the compositions of horizontal arrows
denote the natural projections.

In contrast to standard coverings,
the restrictions $p\colon U'\to U$ are, in general, not homeomorphisms,
but correspond to the projections $G'\backslash\hat U\to G\backslash\hat U$.
Nevertheless, charts $a$ and the respective open subset $U$ of $O$
will be called \emph{evenly covered} by the charts $a'$ and respective open sets $U'$ as above.
Conversely, $a'$ will be called a \emph{lift of $a$} or said \emph{to be above $a$},
and similarly for $U'$ and $U$.
We will also say that $U'$ is a \emph{local leaf of $p$ over $U$}.

Note that coverings are smooth.
An orbifold is said to be \emph{good} if it admits an orbifold covering by a manifold.

\begin{exa}
If $\Gamma\curvearrowright M$ is a properly discontinuous action
of a countable group $\Gamma$ on a manifold $M$ via diffeomorphisms,
then the space $O=\Gamma\backslash M$ of orbits is a good orbifold with $M$ as a covering space.
\end{exa}

A covering $\tilde p\colon\tilde O\to O$ of connected orbifolds with boundary is called \emph{universal}
if, for any covering $p\colon O'\to O$ with $O'$ connected,
and any points $x\in O$, $x'\in O'$, and $\tilde x\in\tilde O$ with $p(x')=\tilde p(\tilde x)=x$,
there is a covering $p'\colon\tilde O\to O'$ such that $\tilde p=p\circ p'$ and $p'(\tilde x)=x'$.

\begin{thm}[Thurston, Proposition 13.2.4 in \cite{Th}]
A connected orbifold with boundary has a universal cover,
and any such cover is unique up to isomorphism.
\end{thm}

For a covering $p\colon O'\to O$ of orbifolds with boundary,
we say that a diffeomorphism $\tau$ of $O'$ is a \emph{deck transformation} if $p\tau=p$.
We say that $p$ is \emph{normal}, if its group of deck transformations is transitive
on the fibers of $p$.
By definition, universal coverings of $O$ are normal.
Up to natural isomorphism,
the groups of deck transformations does not depend on the universal covering
and is called the \emph{fundamental group} of $O$, denoted by $\pi_1^{\orb}(O)$.

\subsection{Riemannian coverings}
\label{susrico}
A covering $p\colon O'\to O$ of Riemannian orbifolds with boundary is called \emph{Riemannian}
if $p$ is a local isometry.

\begin{lem}\label{liftex}
Let $p\colon O'\to O$ be a Riemannian covering of Riemannian orbifolds with boundary,
and assume that $O'$ is complete.
Let $c\colon I\to O$ be a minimizing geodesic, where $I=[a,b]$ or $I=[a,b)$.
Let $x'\in O'$ be a point with $p(x')=c(a)$.
Then there is a lift $c'\colon I\to O'$ with $c'(a)=x'$
such that $L(c'|_{[s,t]})=L(c|_{[s,t]})$ for all $s\le t$ in $I$.
Any such lift $c'$ of $c$ is a minimizing geodesic in $O'$.
\end{lem}

\begin{proof}
We consider the case $I=[a,b]$ first.
We can assume that $a=0$ and that $c$ has unit speed.
Let $A$ be the set of $a\in[0,b]$ such that a corresponding lift exists for $c|_{[0,a]}$.
By the assumption on the length of subcurves, any such lift is of unit speed and minimizing
since $p$ does not increase distances and $c$ is minimizing.
Furthermore, $A$ is not empty since $0\in A$.

We show now that $A$ is closed.
Let $(a_n)$ be an increasing sequence in $A$ with limit $a$
and $c_n\colon[0,a_n]\to O'$ be lifts of $c$ as asserted.
Then for $m\ge n$, $c_m$ is also such a lift on $[0,a_n]$.
Since $O'$ is complete, we may apply the Arzela-Ascoli theorem
and get a subsequence of the $c_m|_{[0,a_n]}$ which converges to a corresponding lift of $c$ on $[0,a_n]$.
We apply this argument again, but now to the subsequence and for $a_m>a_n$.
We get a subsequence of the subsequence of the $c_k|_{[0,a_m]}$ which converges,
and the limit will coincide with the previous limit on $[0,a_n]$.
Iterating this argument, we get a lift of $c$ on $[0,a)$.
By the completeness of $O'$ it can be extended to $[0,a]$, and hence $a\in A$.
Therefore $A$ is closed.

Suppose now that $a\in A$ with $a<b$,
and let $c'\colon[0,a]\to O'$  be a lift of $c|_{[0,a]}$.
Let $(U,\hat U,G,\pi)$ be an evenly covered chart of $O$ with $c(a)\in U$
and $(U',\hat U,G',\pi')$ be the associated chart of $O'$ with $c'(a)\in U'$.
By \cref{hatlift}, there is an $\ve>0$ and a lift $\sigma$ to $\hat U$ of $c|_{[a,a+\ve]}$
such that $L(\sigma)=L(c|_{[a,a+\ve]})$.
Evidently, there exists $g \in G$ with $(\pi' \circ g \circ \sigma) (a) = c'(a)$.
Extending $c'$ to $[a,a+\ve]$ by $\pi'\circ g \circ\sigma$,
we obtain a lift $c'$ of $c$ on $[0,a+\ve]$ such that $L(c')=L(c|_{[0,a+\ve]})$.
As we observed above,
$c'$ is minimizing since $p$ does not increase distances and $c$ is minimizing.

The assertion in the case where $I=[a,b)$ is now by reduction to the case $I=[a,b_n]$,
where the sequence of $b_n$ increases to $b$,
and applying the Arzela-Ascoli theorem as above.  
\end{proof}

The next result is an immediate consequence of \cref{hori}.\ref{hr2}.

\begin{prop}\label{comcov}
For a covering $p\colon O'\to O$ of connected orbifolds with boundary,
a Riemannian metric on $O$ is complete if and only if the lifted Riemannian metric on $O'$ is complete.  
\end{prop}

\begin{proof}
Suppose first that $O'$ is complete, and let $c\colon[0,1)\to O$ be a minimizing geodesic.
Let $c'\colon[0,1)\to O'$ as in \cref{liftex}.
Then \cref{hori}.\ref{hr2} implies that $c'$ can be extended to $[0,1]$,
hence the composition of the extension with $p$ is an extension of $c$ to $[0,1]$.
Hence $O$ is complete.

Suppose now that $O$ is complete and let $c'\colon[0,1)\to O'$ be a minimizing geodesic.
Then $c=p\circ c'$ can be extended to $[0,1]$, by the completeness of $O$.
Then using an evenly covered chart about $c(1)$ and the leaf above it
containing the image of $c'|_{(1-\ve,1)}$ for some $\ve>0$,
we see that $c'$ can be extended to $[0,1]$.
Hence $O'$ is complete.
\end{proof}

\begin{prop}\label{distcov}
Let $p\colon O'\to O$ be a Riemannian covering
of complete and connected Riemannian orbifolds with boundary.
Let $x_0,x_1\in O$ and $y_0\in p^{-1}(x_0)$.
Then there is a point $y_1\in p^{-1}(x_1)$ such that $d(y_0,y_1)=d(x_0,x_1)$.
\end{prop}

Note that $d(y_0,y_1)\ge d(x_0,x_1)$ for all $y_1\in p^{-1}(x_1)$
since $p$ does not increase distances.

\begin{proof}[Proof of \cref{distcov}]
Let $c\colon[0,1]\to O_0$ be a minimizing geodesic from $x_0$ to $x_1$
and $c'$ be a lift of $c$ to $O'$ with $c'(0)=y_0$.
Then $c'$ is minimizing and $p(c'(1))=x_1$.
Hence $y_1=c'(1)\in p^{-1}(x_1)$ and $d(y_0,y_1)=d(x_0,x_1)$.
\end{proof}

\subsection{Dirichlet domains}
\label{susdido}
Consider now a Riemannian covering $p\colon O'\to O$ of complete Riemannian orbifolds without boundary, where $O$ is connected.

\begin{lem}\label{nullnull}
If $N\subseteq O$ has measure zero, then also $p^{-1}(N)\subseteq O'$.
\end{lem}

Fix $x\in O$.
For $y\in p^{-1}(x)$, the \emph{Dirichlet domain} of $p$ centered at $y$ is defined to be
\begin{align}\label{dirdom}
	D_y = \{ z \in O' \mid \text{$d(z,y)\le d(z,y')$ for any $y'\in p^{-1}(x)$} \}.
\end{align}
By \cref{distcov}, $d(y,z)=d(x,p(z))$ for any $z\in D_y$.

\begin{prop}\label{dirdom2}
If $x\in\mathcal R_{O}$ and $y\in p^{-1}(x)$, then
\begin{enumerate}
\item\label{dd1}
$y\in\mathcal R_{O'}$;
\item\label{dd2}
$\partial D_y = \{ z \in O' \mid \text{$d(z,y)=d(z,y')$ for some $y'\ne y$ in $p^{-1}(x)$} \}$;
\item\label{dd3}
$\partial D_y \subseteq p^{-1}(C(x))$ and $\inte(D_y)\subseteq\mathcal R_{O'}$;
\item\label{dd4}
$|D_y\cap p^{-1}(C(x))|_m=0$;
\item\label{dd5}
$p\colon D_y\setminus p^{-1}(C(x))\to O\setminus C(x)$ is an isometry;
\item\label{dd6}
for any integrable function $f$ on $O$, $f\circ p$ is integrable on $D_y$ and
\begin{align*}
	\int_{D_y} f\circ p = \int_{O} f.
\end{align*}
\end{enumerate}
\end{prop}

\begin{proof}
\eqref{dd1} is clear from the definition of coverings of orbifolds.

\eqref{dd2}
Let $z\in D_y$ be such that there is a point $y'\ne y$ in $p^{-1}(x)$ with $d(z,y')=d(z,y)$.
Now there is a minimal geodesic from $y'$ to $z$, by the completeness of $O_1$.
Hence any neighborhood of $z$ contains points which are strictly closer to $y'$ than to $y$.
This shows that the given set belongs to $\partial D_y$.
The converse direction is obvious.

\eqref{dd3}
Let $z\in\partial D_y$ and $y'\in p^{-1}(x)$ be a point with $d(z,y')=d(z,y)$.
If $p(z)$ is singular, then we know from \cref{cutloc} that $p(z) \in C(x)$. If $p(z)$ is regular, consider minimizing geodesics from $y$ to $z$ and $y'$ to $z$.
Since their velocity vectors at $z$ are different,
their projections to $O$ are two different geodesics from $x$ to $p(z)$.
By \cref{distcov}, both are minimal, and hence $p(z)\in C(x)$.
The second assertion follows immediately
from \cref{borz} since $y\in\mathcal R_{O'}$.

\eqref{dd4}
is clear from \cref{cutloc} and \cref{nullnull}.

\eqref{dd5}
We only need to check
that $p\colon D_y\setminus p^{-1}(C(x))\to O\setminus C(x)$ is bijective.
Let $z\ne z'$ be points in $D_y$, and suppose that $p(z)=p(z')$.
Then $p$ maps minimal geodesics from $y$ to $z$ and $z'$ to different geodesics
from $x$ to $u=p(z)=p(z')$.
By \cref{distcov}, they are minimal, and hence $u\in C(x)$.
Therefore $p$ is injective on $D_y\setminus p^{-1}(C(x))$.

Given $u \in O$, let $c \colon [0,1] \to O$ be a minimizing geodesic from $x$ to $u$. Then the lift $c'$ of $c$ starting at $y$ is a minimizing geodesic. Since $p$ does not increase distances, it follows that $d(y,c'(1)) = d(x,u) \leq d(y',c'(1))$ for any $y' \in p^{-1}(x)$, which means that $c'(1) \in D_{y}$. We conclude that $p \colon D_{y} \to O$ is surjective, and so is $p\colon D_y\setminus p^{-1}(C(x))\to O\setminus C(x)$.

\eqref{dd6} is clear from \eqref{dd4} and \eqref{dd5}.
\end{proof}

\subsection{Action of the fundamental group on the fiber}
\label{setup}

Let $p \colon O_1 \to O_0$ be a Riemannian covering of orbifolds with boundary, where $O_0$ is connected.
Consider a (connected) component $O_1'$ of $O_1$.
Then the universal covering $p_0\colon\tilde O\to O_0$ of $O_0$ factors through $O_1'$,
\begin{equation}\label{uncodia}
\begin{tikzcd}
	\tilde O \arrow[r,"p_1'"] \arrow[rd,"p_0"']
	&
	O_1' \arrow[d,"p"]
	\\
	&
	O_0
\end{tikzcd}
\end{equation}
and $p_1'$ is the universal covering of $O_1'$.
In general, the covering $p_1'$ is not unique,
but we fix a choice for each component $O_1'$ of $O_1$.

Recall that the fundamental group $\Gamma_0=\pi_1^{\orb}(O_0)$ of $O_0$
is defined to be the group of deck transformation of $p_0$
and that $\Gamma_0$ is transitive on the fibers of $p_0$
and simply transitive on the fibers over regular points of $O_0$.
The corresponding statements hold for $p_1'$,
where we denote the group of deck transformations of $p_1'$ by $\Gamma_1'$.
Since $p_0=p\circ p_1'$, we have $\Gamma_1'\subseteq\Gamma_0$.

For $x_0\in O_0$, choose $\tilde x\in\tilde O$ with $p_0(\tilde x)=x_0$
and, for each connected component $O_1'$ of $O_1$, set $x_1'=p_1'(\tilde x)\in O_1'$.
Then $p(x_1')=x_0$ and the part of the fiber of $p$ over $x_0$ in $O_1'$ is given by
\begin{align}\label{fiberpar}
	p^{-1}(x_0)\cap O_1' = p_1'(\Gamma_0\tilde x).
\end{align}
If $x_0$ is regular,
we obtain a right action of $\Gamma_0$ on $p^{-1}(x_0)$ by setting
\begin{align}\label{fiberact}
	yg = p_1'(hg\tilde x)
	\hspace{3mm}\text{for}\hspace{3mm}
	y = p_1'(h\tilde x).
\end{align}
Identifying $p_1'(h\tilde x)$ with $\Gamma_1'h$,
this action corresponds to the right action of $\Gamma_0$
on $\Gamma_1'\backslash\Gamma_0$.
Clearly, the action on the fiber depends on the choice of the point $\tilde x$ over $x_0$.

We fix a complete background metric on $O_0$ and its lift to $O_1$
and consider distances and geodesics with respect to these metrics.

Fix $x_0\in\mathcal R_0$ and choose $\tilde x\in\tilde O$ and $x_1'$
in the components $O_1'$ of $O_1$ as above.
For $r > 0$, set
\begin{align}\label{gr}
	G_{r} = \{ g \in \Gamma_0 \mid d(g\tilde{x},\tilde{x}) < r \}
	\hspace{3mm}\text{and}
	\hspace{3mm}
	N(r) = |G_{r}|.
\end{align}
Note that $G_r=G_r^{-1}$.
For any $y_1,y_2\in p^{-1}(x_0)$,
\begin{align}\label{gr2}
	d(y_1,y_2) < r \Longleftrightarrow \text{$y_2=y_1g$ for some $g\in G_r$},
\end{align}
by \cref{distcov}.
For any $\tilde y\in\tilde O$, we have
\begin{align}\label{gr4}
	|\{g\in\Gamma_0 \mid d(g\tilde x,\tilde y)<r \}| \le N(2r),
\end{align}
by the triangle inequality.

\begin{lem}\label{caes}
For any $r>0$ and $y_1\in O_1$, we have
\begin{align*}
	|p^{-1}(x_0)\cap B_{r}(y_1)|\le N(2r).
\end{align*}
\end{lem}

\begin{proof}
Suppose that $y_1\in B_{r}(x'_1g_i)$, $1\le i\le n$, where the $[g_i]$
are pairwise different in $\Gamma_1\backslash\Gamma_0$ .
Then there exist $h_i\in \Gamma_1$ such that
\begin{align*}
	d(h_ig_i\tilde x,\tilde y) = d(x'_1g_i,y_1) < r,
\end{align*}
by \cref{distcov}, where $\tilde y$ is any point in $p_1^{-1}(y_1)$.
Since the $h_ig_i$ are pairwise different, we get $n\le N(2r)$, by \eqref{gr4}.
\end{proof}

\section{Amenability of actions}
\label{secame}

Consider a right action of a countable group $\Gamma$ on a countable set $X$.
The action is called \textit{amenable} if there exists an invariant mean on $\ell^{\infty}(X)$;
that is, a linear map $\mu \colon \ell^{\infty}(X) \to \mathbb{R}$ such that
\begin{align*}
	\inf f \leq \mu (f) \leq \sup f \text{ and } \mu (g^{*} f) = \mu(f)
\end{align*}
for any $f \in \ell^{\infty}(X)$ and any $g \in \Gamma$.
The group $\Gamma$ is called \emph{amenable} if the right action of $\Gamma$ on itself is amenable.

Clearly, any (right) action of $\Gamma $ on any finite set is amenable.
Furthermore, an action of $\Gamma$ on a countable set $X$ is amenable
if its restriction to a non-empty invariant subset of $X$ is amenable.

Amenability refers to some kind of asymptotic smallness of $X$ with respect to the action of $\Gamma$.
This is made precise by the following characterization,
due to F\o{}lner in the case of groups \cite[Main Theorem and Remark]{Fo}
and then extended to actions by Rosenblatt \cite[Theorems 4.4 and 4.9]{Ro}.
Given a finite $G \subseteq \Gamma$ and an $\varepsilon > 0$,
a F\o{}lner set $F$ for $G$ and $\varepsilon$ is a non-empty, finite subset of $X$
satisfying $|Fg \setminus F| < \varepsilon |F|$ for any $g \in G$. 

\begin{thm}\label{Folner}
The action of $\Gamma$ on $X$ is amenable if and only if,
for any finite $G \subseteq \Gamma$ and $\varepsilon > 0$,
there exists a F\o{}lner set for $G$ and $\varepsilon$.
\end{thm}

In particular, it follows that the action of $\Gamma$ on $X$ is amenable
if and only if the restriction to any finitely generated subgroup of $\Gamma$ is amenable. 

Let $p \colon O_{1} \to O_{0}$ be a covering of orbifolds with boundary, where $O_0$ is connected, but $O_1$ possibly not.
As explained in \cref{setup}, for any connected component $O_1'$ of $O_{0}$,
we have for the fundamental groups that $\Gamma_1'\subseteq\Gamma_0$.
The covering $p$ is called \emph{amenable} if the right action of $\Gamma_0$
on the disjoint union of the $\Gamma_1'\backslash\Gamma_0$ is amenable,
where union is taken over all connected components $O_1'$ of $O_{1}$.
After fixing a regular $x_{0} \in O_{0}$, an $\tilde{x} \in \tilde{O}$ over $x_0$,
and a covering $p_1'\colon\tilde{O}\to O_1'$ for any connected component $O_1'$ of $O_{1}$,
the aforementioned action coincides with the action of $\Gamma_{0}$ on $p^{-1}(x_{0})$,
defined in (\ref{fiberact}).
Therefore, the covering $p$ is amenable if and only if the latter action is amenable.

\begin{exa}\label{exancc}
	For any covering $p \colon O_{1} \to O_{0}$ (where $O_0$ is connected),
	\begin{align*}
		p \sqcup \id \colon O_{1} \sqcup O_{0} \to O_{0}
	\end{align*}
	is an amenable covering.
\end{exa}

Recall that F\o{}lner's condition allows us to characterize amenability of an action of a group $\Gamma$
in terms of the restriction of the action to finitely generated subgroups of $\Gamma$.
In the context of coverings,
this is reflected by the following characterization of amenability. 

\begin{prop}\label{amecom}
Let $p \colon O_{1} \to O_{0}$ be a covering with $O_{0}$ connected, and $K_1\subseteq K_2\subseteq \cdots$
be an exhaustion of $O_0$ by compact domains with smooth boundary.
Then $p$ is amenable if and only if the restrictions $p\colon p^{-1}(K_n)\to K_n$ of $p$ are amenable.
\end{prop}

\begin{proof}
Endow $O_{0}$ with a complete Riemannian metric and fix a regular point $x_{0}$ in the interior of $K_{1}$. 
Denote by $p_{0} \colon \tilde{O} \to O_{0}$ and $p_{n} \colon \tilde{K}_{n} \to K_{n}$ the universal coverings,
and choose $\tilde{x} \in p_{0}^{-1}(x_{0})$ and $x_{n} \in p_{n}^{-1}(x_{0})$.
Given $r>0$, consider the finite set $G_{r}$ defined in (\ref{gr}) and the finite sets
\[
	G_{r,n} := \{ g \in \pi_{1}^{\orb}(K_{n}) \mid d(gx_{n},x_{n}) < r \}.
\]
Assume first that $p$ is amenable and let $n \in \mathbb{N}$.
Given $\varepsilon > 0$ and a finite subset $G^{\prime}$ of $\pi_{1}^{\orb}(K_{n})$,
there exists $r > 0$ such that $G \subseteq G_{r,n}$.
Let $F$ be a F\o{}lner set for $G_{r}$ and $\varepsilon/|G_{r}|$.
It follows from (\ref{gr2}) that $d_{p^{-1}(K_{n})}(y,yg^{\prime}) < r$ for any $y \in p^{-1}(x_{0})$ and $g^{\prime} \in G_{r,n}$.
In particular, $d_{O_{1}}(y,yg^{\prime}) < r$ and (\ref{gr2}) yields that there exists $g \in G_{r}$ with $yg^{\prime} = yg$.
Therefore, for any $g'\in G_{r,n}$,
we have that  $Fg'$ is contained in the union of $Fg$ with $g \in G_{r}$, and in particular,
\[
	|Fg' \setminus F| \leq \sum_{g \in G_{r}} |Fg \setminus F| < \varepsilon,
\]
which yields that the covering $p \colon p^{-1}(K_{n}) \to K_{n}$ is amenable.
	
Conversely, consider $\varepsilon > 0$ and a finite subset $G$ of $\pi_{1}^{\orb}(O_{0})$.
Then there exists $r > 0$ such that $G \subseteq G_{r}$
and $n \in \mathbb{N}$ such that $B(x_{0},r) \subseteq K_{n}$.
Since $p \colon p^{-1}(K_{n}) \to K_{n}$ is amenable,
there exists a F\o{}lner set $F$ for $G_{r,n}$ and $\varepsilon/|G_{r,n}|$.
For $y \in F$ and $g \in G_{r}$, we obtain from (\ref{gr2}) that $d_{O_{1}}(yg,y) < r$.
Since $B(x_{0},r) \subseteq K_{n}$,
it follows that any minimizing geodesic from $y$ to $yg$ lies in $p^{-1}(K_{n})$,
and hence, $d_{p^{-1}(K_{n})}(yg,y) < r$.
In view of (\ref{gr2}), we obtain that there exists $g^{\prime} \in G_{r,n}$ such that $yg = yg^{\prime}$,
which means that $Fg$ is contained in the union of $Fg'$ with $g'\in G_{r,n}$.
We conclude that $F$ is a F\o{}lner set for $G$ and $\varepsilon$, which yields that $p$ is amenable.
\end{proof}

\cref{amecom} illustrates the importance of considering non-connected covering spaces.
Namely, the preimage $p^{-1}(K)$ of a compact domain $K$ in $O_0$ with smooth boundary
may not be connected even if $O_1$ is.

\section{Monotonicity of $\lambda_0$}
\label{secele}

In our standard setup of a Riemannian covering $p\colon O_1\to O_0$ of Riemannian orbifolds
with compatible Schr\"odinger operators $S_1 = \Delta + V_1$ and $S_0 = \Delta + V_0$, respectively,
let $f\in C^\infty_c(O_1)$, and define a function $f_0\ge0$ on $O_0$ by
\begin{align}\label{pushd}
	f_0^2(x) =  \sum_{y\in p^{-1}(x)} \frac{|x|}{|y|} f^{2}(y)
\end{align}
on $O_{0}$.
We call $f_0$ the \emph{pushdown} of $f$.

\begin{lem}\label{pushd2}
Given a non-zero $f \in C^{\infty}_{c}(O_1)$,
its pushdown $f_0$ is a Lipschitz function on $O_0$ with compact support
such that $f_0^2$ is smooth and 
\begin{align*}
	R_{S_0}(f_0) \le R_{S_1}(f).
\end{align*}
\end{lem}

Since $\lambda_0(S_0,O_0)$ and $\lambda_0(S_1,O_1)$ are the infimum of Rayleigh quotients
of non-vanishing compactly supported Lipschitz or smooth functions on $O_0$ and $O_1$ (either way),
\eqref{monot} is an immediate consequence of \cref{pushd2}.

\begin{proof}[Proof of \cref{pushd2}]
Let $U=\hat{U}/G$ be an evenly covered domain of $O_0$.
Denote by $V_j\cong\hat U/G_j$, $j\in J$, the connected components of $p^{-1}(U)$,
by $p_j$ the restriction $p|_{V_j}$,
and by $\pi\colon\hat U \to U$ and $\pi_j\colon\hat U\to V_j$ the projections.
Given $x\in U$ and $u\in\pi^{-1}(x)$,
\begin{align*}
\begin{split}
	(f_0^2\circ\pi)(u)
	&= f_0^2(x)
	= \sum_{j \in J} \sum_{y\in p_j^{-1}(x)} \frac{|x|}{|y|} f^2(y) \\
	&= \sum_{j \in J} \sum_{y\in p_j^{-1}(x)} \frac{|x|}{|y|}\frac1{|\pi_j^{-1}(y)|}
		\sum_{v\in\pi_j^{-1}(y)} (f^2\circ\pi_j)(v) \\
	&= \sum_{j \in J} \frac{|x|}{|G_j|} \sum_{v\in\pi^{-1}(x)} (f^2\circ\pi_j)(v) \\
	&= \sum_{j \in J} \frac{1}{|G_j|} \sum_{g\in G} (f^2\circ\pi_j)(gu),
\end{split}
\end{align*}
where we use that $p_j\pi_j=\pi$ and $|y||\pi_j^{-1}(y)|=|G_j|$ for the penultimate equality
and that $|x||\pi^{-1}(x)|=|G|$ for the last.
It follows that $f_0^2\in C^{\infty}_{c}(O_0)$.

For $h\in C^{\infty}(O_0)$, let $h_1=h\circ\pi$ be its lift to $O_1$.
Using that the set $\mathcal R_0$ of regular points of $O_0$ and its preimage in $O_1$ are of full measure
and that the preimage is contained in the set $\mathcal R_1$ of regular points of $O_1$,
we get
\begin{align*}
	\int_{O_0} h f_0^{2}
	= \int_{\mathcal R_0} h f_0^{2}
	= \int_{\mathcal R_0} \Sigma_{y\in p^{-1}(x)} h_1(y)f^2(y)
	= \int_{O_1} h_1 f^{2}.
\end{align*}
Therefore
\begin{align*}
	\| f_0 \|_{L^2(O_0)} = \| f \|_{L^2(O_1)}
	\hspace{3mm}\text{and}\hspace{3mm}
	\int_{O_0} V_0 f_0^2 = \int_{O_1} V_1 f^2.
\end{align*}
Moreover
\begin{align*}
	f_0^2(x) = \sum_{y \in p^{-1}(x)} f^2(y)
\end{align*}
on $\mathcal R_0$,
hence
\begin{align*}
	|\nabla f_0(x)|^2
	\le \sum_{y \in p^{-1}(x)} |\nabla f(y)|^2
\end{align*}
on $\mathcal R_0\cap\{f_0^2\ne0\}$. This shows that $f_{0}$ is non-negative on $O_{0}$,
and smooth with bounded gradient on $\{ f_{0}^{2} \neq 0 \}$, and thus, $f_{0}$ is Lipschitz.
Furthermore, it follows that
\begin{align*}
	\int_{O_{0}} |\nabla f_{0}|^{2}
	\le \int_{O_{1}} |\nabla f|^{2}.
\end{align*}
In conclusion, $R_{S_0}(f_0)\le R_{S_1}(f)$.
\end{proof}

\section{Stability of $\lambda_0$ for amenable coverings}
\label{secsta}

The aim of the section is the proof of \cref{amen}.\ref{tame}.
To that end, fix a complete background metric on $O_0$ and consider its lift to $O_1$. In what follows, distances and geodesics
are taken with respect to the given complete background metrics.
However, gradients, Laplace operators, volumes, and integrals
are taken with respect to the original metrics,
since the main issue of our discussion are Rayleigh quotients
with respect to the original metrics. Fix a regular point $x_{0} \in O_{0}$, a point in the universal covering space above it, and recall the definition of $G_r\subseteq\Gamma_0$ and $N(r)$ from \eqref{gr}.

For $r>0$ and $y\in p^{-1}(x_0)$,
consider the function $\psi_y$ on $O_1$ defined by
\begin{align*}
	\psi_y(z) =
\begin{cases}
	 1 &\text{if $d(z,y)\le r$,}\\
	 r + 1 - d(z,y) &\text{if $r\le d(z,y)\le r+1$,} \\
	 0 &\text{if $d(z,y)\ge r+1$,}
\end{cases}
\end{align*}
a Lipschitz function with Lipschitz constant $1$.
For any $z\in O_{1}$, there are at most $N(2r+3)$ points $y\in p_0^{-1}(y_0)$ with $z\in\supp\psi_y$,
by \cref{caes}.
Hence the function
\begin{align*}
	\psi_{1} = \max \left\{ 0, 1 - \sum_{y \in p^{-1}(x)} \psi_y \right\}
\end{align*}
on $O_1$ is well-defined and admits $N(2r+3)$ as a Lipschitz constant.
Thus we obtain a partition of unity on $O_1$ consisting of
\begin{align}\label{partone}
	\vf_{1} = \frac{\psi_{1}}{\psi_{1} + \sum_{y\in p^{-1}(x)}\psi_y}
	\hspace{3mm}\text{and the}\hspace{3mm}
	\vf_y = \frac{\psi_y}{\psi_1 + \sum_{y\in p^{-1}(x)} \psi_y}
\end{align}
with $y \in p^{-1}(x)$,
the \emph{partition of unity corresponding to} $r > 0$.
Clearly $\supp\vf_y=\supp\psi_y$ and $\sum_{y\in p^{-1}(x)}\vf_y=1$ in $B_{r}(y)$ for any $y\in p^{-1}(x)$.

\begin{lem}\label{partone2}
The functions $\vf_y$ admit Lipschitz constant $3N(2r+3)$,
the function $\vf_1$ admits Lipschitz constant $9N(2r+3)^2$.
\end{lem}

\begin{proof}
The numerator of $\vf_y$ in \eqref{partone} takes values in $[0,1]$ and admits $1$ as a Lipschitz constant,
the denominator takes values in $[1,N(2r+3)]$ and admits $2N(2r+3)$ as a Lipschitz constant.
The first assertion follows now from an easy calculation.
Since $\vf_1=1-\sum\vf_y$, the second is an immediate consequence.
\end{proof}

For a finite subset $P\subseteq p^{-1}(x_0)$ consider the non-negative function
\begin{align}\label{phip}
	\chi = \sum_{y\in P} \vf_y
\end{align}
and the sets
\begin{equation}\label{phiq}
\begin{split}
	Q_+ &= \{ y \in p^{-1}(x) \mid \text{$\chi=1$ in $B_r(y)$} \}, \\
	Q_- &= \{ y \in p^{-1}(x) \mid \text{$0<\chi(z)<1$ for some $z\in B_r(y)$} \}, 
\end{split}
\end{equation}
In virtue of \cref{partone2}, we obtain that $\chi$ admits $3N(2r+3)^{2}$ as a Lipschitz constant.

\begin{prop}\label{amenable}
Suppose that $p\colon O_1\to O_0$ is amenable, and let $\ve>0$.
Then there exists a finite subset $P\subseteq p^{-1}(x_0)$ such that $|Q_-|<\ve|Q_+|$.
\end{prop}

\begin{proof}
	Since the right action of $\Gamma_0$ on $p^{-1}(x_0)$ is amenable,
	there exists a finite subset $P\subseteq p^{-1}(x_0)$ such that $|Pg\setminus P|<\ve|P|$
	for all $g\in G_{2r+2}$.
	
	Let $y\in Q_-$ and $z\in B_{r}(y)$ such that $0<\chi(z)<1$. Since
	\[
	\sum_{y \in p^{-1}(x_{0})} \vf_{y}(z) = 1,
	\]
	it follows that there is $y_0\in P$ and $y_1\in p^{-1}(x_0)\setminus P$, such that $\vf_{y_i} (z) > 0$, $i=0,1$. 
	This yields that $d(y_{i},z) < r+1$, and, in particular, that $d(y_{0},y_{1}) < 2r + 2$.
	In view of \eqref{gr2}, we obtain that there exists $g \in G_{2r+2}$ such that $y_{1} = y_{0} g$,
	which shows that $y_{1} \in Pg \smallsetminus P$, for some $g \in G_{2r+2}$. 
	Hence, there exist at most $\ve N(2r+2) |P|$ such $y_{1}$.
	Since $d(y,y_{1}) < 2r +1$, it follows from \cref{caes} that, for any such $y_{1}$,
	there exist at most $N(4r+2)$ such $y$.
	Therefore, we obtain that
	\[
	|Q_{-}| \le \ve N(4r+2) N(2r+2) |P| \leq \ve N(4r+2) N(2r+2) |Q_- \cup Q_+|,
	\]
	where we use that $P\subseteq Q_-\cup Q_+$.
	\cref{amenable} is an immediate consequence of this inequality.
\end{proof}

\begin{proof}[Proof of \cref{amen}.\ref{tame}]
	Let $f\in C^\infty_c(O_0)$, $f\ne0$, and $f_1=f\circ p$ be the lift of $f$ to $O_1$.
	Choose $x_0\in\mathcal R_0$ and $r>0$ such that $\supp f\subseteq B_{r}(x_0)$.
	Then $f_1$ has support in the neighborhood $U_{r}(p^{-1}(x_0))$ of radius $r$ about $p^{-1}(x_0)$.
	
	Consider the partition of unity on $O_{1}$ corresponding to $r$ as in \eqref{partone}.
	Given $\ve>0$, choose the finite set $P\subseteq p^{-1}(x_0)$ according to \cref{amenable}
	with $\chi$ as in \eqref{phip} and $Q_-$ and $Q_+$ as in \eqref{phiq}.
	Then $\chi$ has compact support contained in the closed neighborhood $N_{r+1}(P)$ of radius $r+1$ about $P$,
	and admits $L_\chi=3N(2r+3)^2$ as a Lipschitz constant.
	
	Since $\supp f \subseteq B_{r}(x_{0})$,
	it is easy to see that $\supp f_{1} \cap D_{y} \subseteq B_{r}(y)$ for any $y \in p^{-1}(x_{0})$.
	This yields that $\supp (\chi f)$ is contained in the union of $D_{y}$ with $y \in Q = Q_{+} \cup Q_{-}$.
	
	We want to extimate the Rayleigh quotient $R_{S_1}(\chi f_1)$.
	Since the intersection of different $D_{y}$'s is of measure zero, we compute
	\begin{align*}
		R_{S_1}(\chi f_1)
		&= \frac{\int_{O_1}\chi f_1S_1(\chi f_1)}{\int_{O_1} (\chi f_1)^2}
		= \frac{\int_{O_1}\{|\nabla(\chi f_1)|^2+\chi f_1V_1\chi f_1\}}{\int_{O_1} \chi^2f_1^2} \\
		&= \frac{\sum_{y \in Q} \int_{D_y}\{|\nabla(\chi f_1)|^2 + V_1\chi^2 f_1^2\}}
		{\sum_{y \in Q}\int_{D_y} \chi^2 f_1^2}\\
		&\le \frac{\sum_{y \in Q} \int_{D_y}\{|\nabla(\chi f_1)|^2
		+ V_1\chi^2 f_1^2\}}{\sum_{y \in Q_{+}}\int_{D_y} \chi^2 f_1^2}.
	\end{align*}
	We now estimate the terms arising from the right hand side.
	Since $\chi=1$ on $B_{y}(y)$ for any $y\in Q_+$,
	we have that $\chi f_{1} = f_{1}$ in a neighborhood of $\supp (\chi f_{1}) \cap D_{y}$.
	This, together with Proposition \ref{dirdom2}.\ref{dd6}, yields that
	\[
	\int_{D_y} \chi^2 f_1^2 = \int_{O_0} f^{2} \text{ and }
	\int_{D_y} \{ |\nabla(\chi f_1)|^2+ V_1\chi^{2} f_1^{2} \}
	= \int_{O_0} \{ |\nabla f|^2+ V f^{2} \}
	\]
	for any $y \in Q_{+}$.
	Denote by $L_{\chi}$ and $L_{f}$ the respective Lipschitz constants of $\chi$ and $f$
	and by $C_{f}$ and $C_{V}$ the respective maximum of $|f|$ and of $|V|$ on $\supp f$.
	It is easy to see that (at any point of $O_{1}$)
	\begin{align*}
		|\nabla(\chi f_1)|^2+ V_1\chi^2 f_1^2
		&\le 2 \chi_{1}^{2} | \nabla f_{1} |^{2} + 2 f_1^{2} | \nabla \chi |^{2} + |V_{1}|\chi^{2} f_{1}^{2} \\
		&\le 2 L_{f}^{2} + 2 C_{f}^{2} L_{\chi}^2 + C_{V} C_{f}^{2} =:C,
	\end{align*}
	where we use that $0 \leq \chi \leq 1$.
	Therefore, using again Proposition \ref{dirdom2}.\ref{dd6}, we obtain that
	\[
	\int_{D_y} \{ |\nabla(\chi f_1)|^2+ V_1\chi^{2} f_1^{2} \} \leq C | \supp f|
	\]
	for any $y \in Q_{-}$. From the above estimates, we conclude that
	\[
	R_{S_1}(\chi f_1) \leq R_{S_0}(f) + \frac{C | \supp f|}{\int_{O_0} f^{2}} \frac{|Q_-|}{|Q_+|}
	< R_{S_0}(f) + \frac{C | \supp f|}{\int_{O_0} f^{2}} \varepsilon.
	\]
	This shows that, for any $\delta>0$,
	we have $R_{S_1}(\chi f_1)=R_{S_0}(f) + \delta$ if $\ve$ is chosen sufficiently small. The proof is completed by (\ref{bspec}).
\end{proof}

\section{Stability implies amenability: the case of a closed base}
\label{secclos}

Let $p \colon O_{1} \to O_{0}$ be a Riemannian covering with $O_{0}$ closed (that is, compact and without boundary) and connected,
and $O_{1}$ possibly non-connected.
The aim of this section is to prove the following:

\begin{thm}\label{amecc}
	If $\lambda_{0}(O_{1}) = 0$, then $p$ is amenable.
\end{thm}

We have the right action of $\pi_{1}^{\rm orb}(O_{0})$ on the right cosets of the fundamental group of any connected component of $O_{1}$ in $\pi_{1}^{\rm orb}(O_{0})$,
and, by definition, amenability of $p$ means that this action on the disjoint union of these cosets is amenable.
Recall that this action is equivalent to the action on the fiber of $p$, presented in \cref{setup}.
After fixing a regular point $x_{0} \in O_{0}$ and a point in the universal covering space above it as in \cref{setup}, for any $y_{1},y_{2} \in p^{-1}(x_{0})$, we have that $d(y_{1},y_{2}) < r$ if and only if $y_{2} = y_{1}g$ for some $g \in G_{r}$.

Cover $O_{0}$ with finitely many evenly covered, coordinate systems $\pi_{i} \colon \hat{U}_{i} \to U_{i} = \hat{U}_{i}/G_{i}$, which are extensible; that is, each $\pi_{i}$ can be extended to an evenly covered, coordinate system $\pi_{i} \colon \hat{V}_{i} \to V_{i}$, where $\hat{V}_{i} \subseteq \mathbb{R}^{n}$ is bounded and the closure of $\hat{U}_{i}$ contained in $\hat{V}_{i}$. Since $O_{0}$ is compact, there exists $r_{0}>0$ such that for any $x \in O_{0}$, we have that $B_{3r_{0}}(x)$ is contained in some $U_{i}$. It should be noticed that for any $y \in O_{1}$ there exists a lifted coordinate system $\pi_{ij} \colon \hat{U}_{i} \to V_{ij} = \hat{U}_{i}/G_{ij}$ such that $B_{3r_{0}}(y) \subseteq V_{ij}$.

Since $O_{0}$ is closed, it has Ricci curvature bounded from below, by $1-m$, say, and so does $O_{1}$.

\begin{lem}\label{isop est}
There exists a constant $C > 0$ such that, for any $i$, $x \in U_{i}$ and $\hat{x} \in \pi_{i}^{-1}(x)$, we have that
\begin{align*}
	h^{N}(B_{r_{0}}(\hat{x})) \geq C \text{ and } h^{N}(B_{2r_{0}}(\hat{x})) \geq C.
\end{align*}
\end{lem}

\begin{proof}
We may extend the Riemannian metrics on the $\hat U_i$ to complete Riemannian metrics on $\mathbb{R}^{n}$
such that there Ricci curvature is bounded from below.
Then we can apply \cref{bulem} to arrive at \cref{isop est}.
\end{proof}

Recall that a subset $X$ of an orbifold $O$ is called a complete $2r$-separated subset
if $X$ is a maximal subset with the property $d(x,y) \geq 2r$ for any $x \neq y$ in $X$. It is clear that if $X$ is a complete $2r$-separated subset of $O$, then the balls $B_{2r}(x)$ with $x \in X$ cover $O$.

\begin{cor}\label{cover}
	There exists $C(r_{0}) > 0$ such that for any complete $2r_{0}$-package $X$ of $O_{1}$, we have that any $x \in O_{1}$ belongs to at most $C(r_{0})$ of the balls $B_{2r_{0}}(y)$ with $y \in X$.
\end{cor}

\begin{proof}
	Given a complete $2r_{0}$-package $X$ and $x \in O_{1}$, set $E_{x} := \{ y \in X : x \in B_{2r_{0}}(y)\}$. It is evident that the disjoint balls $B_{r_{0}}(y)$ with $y \in E_{x}$, are contained in $B_{3r_{0}}(x)$. It should be observed that $|B_{r_{0}}(y)| \geq |B_{r_{0}}(p(y))| \geq c > 0$, since $O_{0}$ is closed. We conclude from Proposition \ref{bgv} that
	\[
	c |E_{x}| \leq \sum_{y \in E_{x}} |B_{r_{0}}(y)| \leq  |B_{3r_{0}}(x)| \leq \frac{1}{|\Gamma_{x}|} \beta(3r_{0}) \leq \beta(3r_{0}),
	\]
	as we wished.
\end{proof}

We are ready to prove an analogue of a special version of Buser's \cite[Lemma 7.2]{Bu}.

\begin{prop}\label{Buser}
	If $h(O_{1}) = 0$, then for any $\varepsilon,r > 0$, there exists open bounded $A \subseteq O_{1}$ such that $|U_{r}(\partial A)| < \varepsilon |A|$.
\end{prop}

\begin{proof}
	In view of the volume comparison theorem, it suffices to prove the assertion for any $\varepsilon > 0$ and a fixed $r > 0$. Set $r = r_{0}$ from the beginning of this section, and
	\[
	C_{0} := \max_{i} |G_{i}|,
	\]
	where $G_{i}$ are the groups corresponding to the coordinate systems in the beginning of this section. Since $h(O_{1}) = 0$, we have that for any $\varepsilon > 0$, there exists a smoothly bounded, compact domain $A \subseteq O_{1}$ with
	\begin{align}\label{a}
		\frac{|\partial A|}{|A|} < \delta := \min \left\{ \frac{C \beta(r) \varepsilon}{2 C_{0}^{2} \beta(4r)} , \frac{C \beta(r) \varepsilon }{2 C(r) C_{0} \beta(2r)} \right\}
	\end{align}
	We partition $O_{1}$ into the sets
	\begin{align}
		A_+ &= \{ x\in O_{1} \mid |A\cap B_r(x)| > \frac1{2C_0}|B_r(x)| \}, \label{ax1} \\
		A_0 &= \{ x\in O_{1} \mid |A\cap B_r(x)| = \frac1{2C_0}|B_r(x)| \}, \label{ax2} \\
		A_- &= \{ x\in O_{1} \mid |A\cap B_r(x)| < \frac1{2C_0}|B_r(x)| \}. \label{ax3} 
	\end{align}
	Clearly, $|A\cap B_r(x)|\ne0$ for all $x\in A_+\cup A_0$.
	Since $|B_r(x)|$ and $|A\cap B_r(x)|$ depend continuously on $x$,
	a path from $A_-$ to $A_+$ will pass through $A_0$.
	Since $A$ is bounded, $A_+$ and $A_0$ are bounded.
	Moreover, $\partial A_+\subseteq A_0$, $A_+$ and $A_-$ are open, and $A_0$ is closed, hence compact. We will show that $A_+$ satisfies the asserted inequality. By passing from $A$ to $A_+$, we get rid of a possibly \lq\lq hairy structure\rq\rq{} along the \lq\lq outer part\rq\rq{} of $A$. We pay by possibly losing regularity of the boundary.
	
	We now choose a $2r$-separated subset $X$ of $O_{1}$ as follows.
	We start with a $2r$-separated subset $X_0\subseteq A_0$ such that $A_0$ is contained in the union of the balls $B_{2r}(x)$ with $x\in X_0$.
	(If $A_0=\emptyset$, then $X_0=\emptyset$.)
	We extend $X_0$ to a $2r$-separated subset $X_0\cup X_+$ of $A_0\cup A_+$ such that $A_0\cup A_+$ is contained in the union of the balls $B_{2r}(x)$ with $x\in X_0\cup X_+$.
	(If $A_+=\emptyset$, then $X_+=\emptyset$.)
	We finally extend $X_0\cup X_+$ to a complete $2r$-separated subset $X=X_0\cup X_+\cup X_-$ of $O_{1}$.
	(If $A_-=\emptyset$, then $X_-=\emptyset$.)
	By definition, $X_+\subseteq A_+$ and $X_-\subseteq A_-$.
	Since $A$ is bounded and $|A\cap B_{r}(x)|\ne0$ for all $x\in X_0\cup X_+$, the sets $X_0$ and $X_+$ are finite.
	By the same reason, the set $Y$ of $x\in X_-$ with $|A\cap B_{2r}(x)| \ne0$ is finite.
	
	The neighborhood $U_{2r}(A_0)$ is covered by the balls $B_{4r}(x)$ with $x\in X_0$.
	Using Proposition \ref{bgv} and \eqref{ax2}, we therefore get
	\begin{align*}
		|U_{2r}(A_0)|
		&\le \sum_{x\in X_0} |B_{4r}(x)| \\
		&\le \frac{\beta(4r)}{\beta(r)} \sum_{x\in X_0} |B_{r}(x)| \\
		&=  \frac{2 C_{0} \beta(4r)}{\beta(r)} \sum_{x\in X_0} |A\cap B_{r}(x)|.
	\end{align*}
	
	For any $x \in X_{0}$ there exists a (lifted) coordinate system $\pi_{ij} \colon \hat{U}_{i} \to V_{ij} = \hat{U}_{i}/G_{ij}$ with $B_{3r}(x) \subseteq V_{ij}$. Fixing $\hat{x} \in \pi_{ij}^{-1}(x)$, we compute 
	\[
	|\pi_{ij}^{-1}(A) \cap B_{r}(\hat{x})| \leq |\pi_{ij}^{-1}(A \cap B_{r}(x))| = |G_{ij}| |A \cap B_{r}(x)| = \frac{|G_{ij}|}{2C_{0}} |B_{r}(x)|.
	\]
	It is easily checked that
	\[
	|B_{r}(x)| = \frac{1}{|G_{ij}|} |\pi_{ij}^{-1} (B_{r}(x))| \leq  \frac{1}{|G_{ij}|} \sum_{z \in \pi_{ij}^{-1}(x)} |B_{r}(z)| \leq |B_{r}(\hat{x})|,
	\]
	which shows that
\begin{align}\label{est}
	|\pi_{ij}^{-1}(A) \cap B_{r}(\hat{x})| \leq \frac{1}{2} |B_{r}(\hat{x})|.
\end{align}
From Lemmas \ref{isop est} and \ref{mdba} we derive that
	\begin{align}\label{isop lift}
		\frac{|\partial A \cap B_{r}(x)|}{|A \cap B_{r}(x)|} \geq \frac{1}{|G_{ij}|} \frac{| \pi_{ij}^{-1}(\partial A) \cap B_{r}(\hat{x})|}{|\pi_{ij}^{-1}(A) \cap B_{r}(\hat{x})|} \ge \frac{C}{C_{0}}
	\end{align}
	for any $x \in X_{0}$, where we used that $G_{ij}$ is a subgroup of $G_{i}$. Hence
	\begin{equation}\label{ax5}
	\begin{split}
	|U_{2r}(A_0)| 
	&\le \frac{2 C_{0}^{2}\beta(4r)}{C\beta(r)} \sum_{x\in X_0} |\partial A\cap B_{r}(x)| \\
	&\le \frac{2 C_{0}^{2}\beta(4r)}{C\beta(r)} |\partial A| \\
	&\le \frac{2 C_{0}^{2}\beta(4r)}{C\beta(r)} \delta |A| \le \varepsilon |A|
	\end{split}
	\end{equation}
	where we use that $A$ satisfies \eqref{a}.
	
	Since any curve from $A_+$ to $A_-$ passes through $A_0$, $A_+$ has distance at least $2r$ to $A_-\setminus U_{2r}(A_0)$.
	Hence $A_-\setminus U_{2r}(A_0)$ is covered by the open balls $B_{2r}(x)$ with $x\in X_-$. 
	
	With $Y$ as above, we let $Z=X_0\cup Y$. Again, for any $x \in Z$ there exists a (lifted) coordinate system $\pi_{ij} \colon \hat{U}_{i} \to V_{ij} = \hat{U}_{i}/G_{ij}$ with $B_{3r}(x) \subseteq V_{ij}$. Arguing as above, using \eqref{ax2} and \eqref{ax3}, we readily see that
	\[
	|\pi_{ij}^{-1}(A)\cap B_{r}(\hat{x})| \leq \frac12 |B_{r}(\hat{x})|
	\]
	for any $\hat{x}\in \pi_{ij}^{-1}(x)$. Letting $A^c=O_{1}\setminus A$, we obtain from Proposition \ref{bgv} that
	\begin{equation*}
		\begin{split}
			|\pi_{ij}^{-1}(A)^c\cap B_{2r}(\hat{x})|
			&\ge |\pi_{ij}^{-1}(A)^c\cap B_{r}(\hat{x})| \ge \frac12 |B_{r}(\hat{x})| \\
			&\ge \frac{\beta(r)}{2\beta(2r)}|B_{2r}(\hat{x})| \ge  \frac{\beta(r)}{2\beta(2r)}|\pi_{ij}^{-1}(A) \cap B_{2r}(\hat{x})| > 0.
		\end{split}
	\end{equation*}
	for any $x\in Z$. With the constant $C$ from Lemma \ref{isop est}, we therefore get
	\begin{equation}\label{az4}
	\begin{split}
	C
	&\le h^N(B_{2r}(\hat{x})) \\
	&\le \frac{|\pi_{ij}^{-1}( \partial A) \cap B_{2r}(\hat{x})|}{\min\{|\pi_{ij}^{-1}(A)\cap B_{2r}(\hat{x})|,|\pi_{ij}^{-1}(A)^c\cap B_{2r}(\hat{x})|\}} \\
	&\le \frac{2\beta(2r)}{\beta(r)} \frac{|\pi_{ij}^{-1}( \partial A) \cap B_{2r}(\hat{x})|}{|\pi_{ij}^{-1}(A)\cap B_{2r}(\hat{x})|} \\
	&\le \frac{2 C_{0} \beta(2r)}{\beta(r)} \frac{|\partial A\cap B_{2r}(x)|}{|A\cap B_{2r}(x)|}
	\end{split}
	\end{equation}
	for any $x\in Z$, and the last inequality follows similarly to \eqref{isop lift}. Using Corollary \ref{cover}, \eqref{az4} and \eqref{a}, we conclude that
	\begin{equation}\label{az5}
	\begin{split}
	|A\cap(A_-\setminus U_{2r}(A_0))|
	&\le \sum_{x\in Z}|A\cap B_{2r}(x)| \\
	&\le \frac{2 C_{0}\beta(2r)}{C \beta(r)} \sum_{x\in Z}|\partial A\cap B_{2r}(x)| \\
	&\le \frac{2 C(r) C_{0}\beta(2r)}{C \beta(r)}  |\partial A| \\
	&< \frac{2 C(r) C_{0}\beta(2r)}{C \beta(r)} \delta |A| \le \varepsilon |A|,
	\end{split}
	\end{equation}
	where we use \eqref{a} in the last step.
	
	Since $A\subseteq A_+\cup U_{2r}(A_0)\cup(A\cap(A_-\setminus U_{2r}(A_0)))$, we obtain
	\begin{align*}
		|A_+|
		&\ge |A| - |U_{2r}(A_0)| - |A\cap(A_-\setminus U_{2r}(A_0))| \\
		&\ge (1-2\ve)|A|.
	\end{align*}
	In particular, $A_+$ is not empty. Since $\partial A_{+} \subseteq A_{0}$, we conclude that
	\begin{align*}
		|U_{2r}(\partial A_+)| \le |U_{2r}(A_0)| \le \ve|A|
		\le \frac{\ve}{1-2\ve}|A_+|.
	\end{align*}
	In conclusion, $A_+$ is a bounded open subset of $O_{1}$ that satisfies the asserted inequality, albeit with $2\ve$ in place of $\ve$ (assuming w.l.o.g.\;that $\ve<1/4$). 
\end{proof}

We are now ready to prove the main result of the section. As a consequence of the Cheeger inequality, if $\lambda_{0}(O_{1}) = 0$, then $h(O_{1}) = 0$. We know from Lemma \ref{Buser} that for any $\varepsilon > 0$ and $r > 2 \diam O_{0}$ there exists an open, bounded $A \subseteq O_{1}$ with $|U_{3r}(A) \setminus A| \leq | U_{3r}(\partial A) | < \ve |A|$.

Fix a regular point $x \in O_{0}$ and consider the finite set $F := p^{-1}(x) \cap U_{r}(A)$. 
Taking into account that
\[
\diam D_{y} \leq 2 \diam O_{0} < r,
\]
it is immediate to verify that $A$ is contained in the union of $D_{y}$, with $y \in F$.
Moreover, given $g \in G_{r}$, it is easy to see that any $y \in F g \setminus F$ belongs to $U_{2r}(A) \setminus U_{r}(A)$, which shows that $U_{3r}(A) \setminus A$ contains the corresponding $D_{y}$. Using that $|D_{y}| = |O_{0}|$ and that the intersection of different $D_{y}$'s is measure zero, we conclude that
\[
|Fg \setminus F| |O_{0}| \leq |U_{3r}(A) \setminus A| < \varepsilon |A| \leq \varepsilon |F| |O_0|
\]
for any $g \in G_{r}$. Since any finite subset $G$ of $\pi_{1}^{\rm orb}(O_{0})$ is contained in $G_{r}$ for some $r > 2 \diam O_{0}$, this implies that the covering is amenable.

\section{Stability implies amenability: the role of $\lambda_{\ess}$}
\label{secgene}
Let $K_0$ be a compact and connected Riemannian orbifold with non-empty boundary.
Let $p\colon K_1\to K_0$ be a Riemannian covering of orbifolds,
where we do not assume that $K_1$ is connected.
We assume that $\lambda_0(K_1)=\lambda_0(K_0)=0$,
where we recall the notation $\lambda_0(K)=\lambda(\Delta,K)$,
where we use the definition of $\lambda_0$ as the infimum of the usual Rayleigh quotients
over non-zero functions $f$ in $C^\infty_c(K_1)$ and $C^\infty_c(K_0)$,
respectively.
Note that we do not require that the functions $f$ vanish on the corresponding boundary.

Change the given Riemannian metric on $K_0$ in a neighborhood $U\cong[0,\ve)\times\partial K_0$
of $\partial K_0$ so that the new metric is a product metric $dr^2+g_0$ on $U$,
where $g_0$ is a Riemannian metric on $\partial K_0$,
and endow $K_1$ with the lifted metric.
Since $K_0$ is compact, the old and new Riemannian metrics on $K_0$ and $K_1$ are uniformly equivalent,
and hence $\lambda_0(K_1)=\lambda_0(K_0)=0$ with respect to the new metric.

Denote by $2K_0$ and $2K_1$ the doubles of $K_0$ and $K_1$.
Since the new Riemannian metrics above are product metrics in neighborhoods of the boundaries,
they fit together to define Riemannian metrics on  $2K_0$ and $2K_1$
so that $p$ extends to a Riemannian covering $2p\colon2O_1\to2O_0$.
Since $\lambda_0(K_1)=0$ with respect to the new metric
and test functions in $C^{\infty}_c(K_1)$ can be doubled to test functions in $\Lip_c(2K_1)$
with the same Rayleigh quotient, we get that $\lambda_0(2K_1)=0$.
Since $2K_0$ is closed, we conclude from \cref{amecc} that the covering $2p$ is amenable. It follows from Proposition \ref{amecom} that the restriction of $2p$ over $K_{0}$, which is the original covering $p$, is amenable. Hence, we arrive at the following:

\begin{thm}\label{amecb}
If $\lambda_0(K_1)=0$, then $p$ is amenable.
\end{thm}

\begin{proof}[Proof of \cref{amen}.\ref{name}]
Assume that
\begin{align*}
		\lambda_{0}(S_1,O_1) = \lambda_0(S_0,O_0)
\end{align*}
and, to arrive at a contradiction, that $p$ is non-amenable.
According to \eqref{bspec},
there exists a sequence $(f_{n})_{n \in \mathbb{N}}$ in $C^{\infty}_{c}(O_{1})$
with $L^2$-norm one and $R_{S_1}(f_{n}) \rightarrow \lambda_{0}(S_1,O_1)$.
Since the covering is non-amenable,
we obtain from \cref{amecom} that there exists a smoothly bounded, compact domain $K \subseteq O_0$ 
such that the covering $p\colon p^{-1}(K) \to K$ is non-amenable.
Then \cref{amecb} implies that $\lambda_0(p^{-1}(K)) > 0$.
	
We know from \cref{eigen} that there exists a positive $\vf_{0} \in C^{\infty}(O_{0})$
with $L^{2}$-norm one and $S_{0} \vf_{0} = \lambda_{0}(S_{0},O_{0}) \vf_{0}$.
Denote by $\vf_{1}$ the lift of $\vf_{0}$ to $O_{1}$ and by $L$ the renormalization of $S_1$
with respect to $\vf_{1}$.
Since $\vf_{1}$ is positive, we may write $f_{n} = h_{n} \vf_{1}$, and in view of Lemma \ref{renorm} we have that
	\[
	R_{L}(h_{n}) = \frac{\int_{O_{1}} | \nabla h_{n} |^{2} \vf_{1}^{2} }{\int_{O_{1}} h_{n}^{2} \vf_{1}^{2}} = R_{S_{1}}(f_{n}) - \lambda_{0}(S_{0},O_{0}) \rightarrow 0.
	\]
	Denoting by $c_{1}, c_{2} > 0$ the minimum and the maximum of $\vf_{0}$ on $K$, respectively, we have that
	\[
	\frac{\int_{p^{-1}(K)} | \nabla h_{n}|^{2} \vf_{1}^{2}}{\int_{p^{-1}(K)} h_{n}^{2} \vf_{1}^{2}}
	\geq \frac{c_{1}^{2}}{c_{2}^{2}} \lambda_{0}(p^{-1}(K)) > 0,
	\]
	which shows that
	\[
	\int_{p^{-1}(K)} h_{n}^{2} \vf_{1}^{2} \rightarrow 0
	\hspace{3mm}\text{and}\hspace{3mm}
	\int_{O_{1} \smallsetminus p^{-1}(K)} h_{n}^{2} \vf_{1}^{2} \rightarrow 1.
	\]
Let $K_{0}$ be a compact domain (of positive measure) in the interior of $K$
and consider $\chi_{0} \in C^{\infty}_{c}(O_{1})$ with $\chi_{0} = 1$ in a neighborhood of $K_{0}$
and $\supp \chi_{0} \subseteq K$.
Denote by $\chi_{1}$ the lift of $\chi_{0}$ to $O_{1}$ and set $h_n' = h_n(1 - \chi_1)$. 
It is not difficult to verify that
	\[
	\int_{O_{1}} (h_{n}^{\prime})^{2} \vf_{1}^{2} \rightarrow 1
	\hspace{3mm}\text{and}\hspace{3mm}
	\int_{O_{1}} |\nabla h_{n}^{\prime}|^{2} \vf_{1}^{2} \rightarrow 0,
	\]
and hence $R_{L}(h_{n}^{\prime}) \rightarrow 0$.
Now setting $f_n'=h_n'\vf_{1}$, we derive from \cref{renorm} that $R_{S_1}(f_n')\to\lambda_0(S_0,O_0)$.
It should be noticed that $\| f_n'\|_{L^{2}} \rightarrow 1$ and $\supp f_n'\cap p^{-1}(K_{0}) = \emptyset$.
	
It is easy to see that the sequence $(g_n)_{n \in\N}$ in $\Lip_{c}(O_0)$, 
	consisting of the pushdowns of $f_{n}^{\prime}$ as defined in \eqref{pushd},
	satisfies $\| g_{n} \|_{L^2} \rightarrow 1$, $\supp g_n\cap K_{0} = \emptyset$, 
	and $R_{S_0}(g_{n}) \rightarrow \lambda_{0}(S_{0},O_{0})$.
Now we see the role of $\lambda_{\ess}$.
Namely, by \cref{min seq}, the assumption that $\lambda_{0}(S_{0},O_{0}) < \lambda_{\ess}(S_{0},O_{0})$
yields that $g_{n} \rightarrow \varphi_{0}$ in $L^{2}(O_0)$,
	after passing to a subsequence if necessary.
	This is a contradiction since $\varphi_{0}$ is positive,
	whereas $K_{0}$ has positive measure and $\supp g_n\cap K_{0} = \emptyset$.
\end{proof}

\section{Conformally compact orbifolds}
\label{seccc}

Let $O=P\setminus\partial P$, $\partial P=\{\rho=0\}$, and $g=h/\rho^2$ as in the introduction.
Then the normalized gradient field $X=\nabla\rho/|\nabla\rho|$ of $\rho$ with respect to $h$
is well defined in a neighborhood of $\partial P$,
and $V=\rho X$ is the normalized gradient field of $\rho$ with respect to $g$.
The divergence of $V$ with respect to $g$ is given
\begin{align*}
	\dive_gV
	&= \frac{m}2V(\ln(1/\rho^2)) + \dive V \\
	&= \frac{m}2\frac\rho{|\nabla\rho|}d(\ln(1/\rho^2)(\nabla\rho) + \dive V \\
	&= - m|\nabla\rho| + \dive V \\
	&=  -(m-1)|\nabla\rho| + \rho\dive X.
\end{align*}

\begin{proof}[Proof of \eqref{coco}]
Since $\dive X$ is a smooth function in a neighborhood of $\partial P$ and $\partial P=\{\rho=0\}$,
we conclude that, for any $\ve>0$, there is a neighborhood $U$ of $\partial P$ such that
\begin{align*}
	\dive_gV = -(m-1)|\nabla\rho| \pm \ve \ge  -a(m-1) - \ve
\end{align*}
in $U$.
By the divergence formula, we have
\begin{align*}
	\min\dive V |D|_m \le \int_D\dive V
	= \int_{\partial D} \la V,\nu\ra \le |\partial D|_{m-1}
\end{align*}
for any compact domain in $U$ with smooth boundary and hence
\begin{align*}
	h_{\ess}(O) \ge (m-1)a.
\end{align*}
Using the Cheng eigenvalue comparison \cref{cheng} for orbifolds,
the proof of the inequalities $\lambda_{\ess}(\tilde O),\lambda_{\ess}(O)\le a^2(m-1)^2/4$
is the same as that for the corresponding inequalities for $\lambda_0$
in the case of manifolds in \cite[Theorem 1.10]{BMP2}.
The Cheeger inequality \eqref{china} then implies
the asserted equality $\lambda_{\ess}(O)=a^2(m-1)^2/4$.
\end{proof}


\end{document}